\documentclass[11pt,reqno]{amsart}
\usepackage{amsmath,amssymb,amsfonts}
\usepackage{enumerate}
\usepackage[nosumlimits]{mathtools}
\usepackage{graphicx}
\usepackage{subcaption}
\usepackage{algorithm}
\usepackage{algpseudocode}
\usepackage[margin=1in]{geometry}
\usepackage{xcolor}
\usepackage{adjustbox} 
\usepackage[scr=rsfs,cal=boondox]{mathalfa}
\usepackage{tikz-cd}
\usepackage{tikz}
\usetikzlibrary{calc,positioning}
\usetikzlibrary{shapes,arrows}
\usetikzlibrary{decorations.pathmorphing, decorations.text}
\usetikzlibrary{decorations.markings,intersections}
\usepackage{hyperref,xcolor}
\hypersetup{
    colorlinks,
    linkcolor={red!50!black},
    citecolor={blue!50!black},
    urlcolor={blue!80!black}
}

\theoremstyle{plain}
\newtheorem{theorem}{Theorem}[section]
\newtheorem{proposition}[theorem]{Proposition}
\newtheorem{lemma}[theorem]{Lemma}
\newtheorem{corollary}[theorem]{Corollary}
\theoremstyle{definition}
\newtheorem{definition}[theorem]{Definition}
\newtheorem{example}[theorem]{Example}
\newtheorem{remark}[theorem]{Remark}
\renewcommand{\algorithmicrequire}{\textbf{Input:}}
\renewcommand{\algorithmicensure}{\textbf{Output:}}
\newcommand{\tp}{{\scriptscriptstyle\mathsf{T}}}
\DeclareMathOperator{\GL}{GL}

\DeclareMathOperator{\rank}{rank}
\DeclareMathOperator{\Gr}{Gr}
\DeclareMathOperator{\proj}{proj}
\DeclareMathOperator{\Herm}{Herm}

\DeclareMathOperator{\diag}{diag}
\DeclareMathOperator{\tr}{tr}
\tikzset{
  symbol/.style={
    draw=none,
    every to/.append style={
      edge node={node [sloped, allow upside down, auto=false]{$#1$}}}
  }
}
\tikzstyle{startstop} = [rectangle, rounded corners, minimum width = 2cm, minimum height=1cm,text centered, draw = black, fill = red!40]

\tikzstyle{process} = [rectangle, minimum width=1cm, minimum height=1cm, text centered, draw=black, fill = yellow!50]

\tikzstyle{processb} = [rectangle, minimum width=1cm, minimum height=1cm, text centered, draw=black, fill = blue!50]

\tikzstyle{arrow} = [->,>=stealth]

\tikzstyle{dotsnd}=[rectangle, minimum width=1cm, minimum height=1cm, text centered]

\tikzstyle{processd} = [dashed, minimum width=8cm, minimum height=1.5cm, text centered, draw=black]

\tikzstyle{processdd} = [dashed, minimum width=7cm, minimum height=1.5cm, text centered, draw=black]
\begin{document}
\title{Bundle-based similarity measurement for positive semidefinite matrices}
\author[P.~Liu]{Peng Liu}
\address{KLMM, Academy of Mathematics and Systems Science, Chinese Academy of Sciences, Beijing 100190, China}
\email{liupeng9@amss.ac.cn}
\author[K.~Ye]{Ke Ye}
\address{KLMM, Academy of Mathematics and Systems Science, Chinese Academy of Sciences, Beijing 100190, China}
\email{keyk@amss.ac.cn}
\date{}
\begin{abstract}
Positive semidefinite (PSD) matrices are indispensable in many fields of science. A similarity measurement for such matrices is usually an essential ingredient in the mathematical modelling of a scientific problem. This paper proposes a unified framework to construct similarity measurements for PSD matrices. The framework is obtained by exploring the fiber bundle structure of the cone of PSD matrices and generalizing the idea of the point-set distance previously developed for linear subsapces and positive definite (PD) matrices. The framework demonstrates both theoretical advantages and computational convenience: \emph{(1)} We prove that the similarity measurement constructed by the framework can be recognized either as the cost of a parallel transport or as the length of a quasi-geodesic curve. \emph{(2)} We extend commonly used divergences for equidimensional PD matrices to the non-equidimensional case. Examples include Kullback-Leibler divergence, Bhattacharyya divergence and R\'{e}nyi divergence. We prove that these extensions enjoy  the same consistency property as their counterpart for geodesic distance. \emph{(3)} We apply our geometric framework to further extend those in \emph{(2)} to similarity measurements for arbitrary PSD matrices. We also provide simple formulae to compute these similarity measurements in most situations.
\end{abstract}
\maketitle
\section{Introduction}
Positive semidefinite (PSD) matrices are widespread in numerous scientific fields under different names. For instance, they are called covariance matrices in data analysis \cite{PE08,YWD22}, kernel matrices in machine learning \cite{LCBGJ04,ML23}, genetic variance–covariance matrices in genetic envolution \cite{Pigliucci06,AHMB14}, correlation matrices in wireless communication \cite{HB23,KSVR03} and density matrices in quantum computation and quantum mechanics \cite{NC11,White92}. Due to its great importance in applications, the class of PSD matrices has been studied by mathematicians from various perspectives \cite{GV13,VB96,NRS10,Bhatia15,BS10,lim2019geometric,AJA22,CL22}. One of the most fundamental problems concerning PSD matrices is: \emph{how to measure the similarity between two such matrices?} 

A distance function on the space of PSD matrices is arguably the most natural similarity measurement. Obviously, the restriction of any matrix distance to the space of PSD matrices is one such candidate. However, this induced distance fails to reflect the intrinsic geometry of PSD matrices. For example, we consider the following PSD matrices of rank one:  
\[
A = \begin{bmatrix}
1 & 0 \\
0 & 0
\end{bmatrix},\quad  B = \begin{bmatrix}
0 & 0 \\
0 & 1
\end{bmatrix}. 
\] 
The geodesic curve connecting $A$ and $B$ in $\mathbb{R}^{2\times 2}$ with respect to the Frobenius distance is simply 
\[
\gamma(t) \coloneqq t A + (1 - t) B = \begin{bmatrix}
t & 0 \\
0 & 1-t
\end{bmatrix},\quad  0 \le t \le 1.
\]
Clearly, $\rank \left( \gamma(t) \right) = 2$ for all $t \in (0,1)$. This failure of rank-preserving leads to undesirable effects in applications \cite{PFA06,AFPA06}. To circumvent the problem, geometric techniques are introduced to obtain more intrinsic distances. It is noticed that the space $\Herm_n^{++}$ of all $n\times n$ positive definite matrices is a smooth manifold. Thus, different geodesic distances can be obtained if we equip $\Herm_n^{++}$ with different Riemannian metrics \cite{FD09,FD12,Bhatia15}. Moreover, the quotient manifold structure of $\Herm_{n,r}^{+}$ is fully explored in the literature to obtain distance functions \cite{MA20,MHA19,BS10}. Here $\Herm_{n,r}^{+}$ denotes the space of all $n\times n$ PSD matrices of the fixed rank $r$. Unfortunately, since the space $\bigcup_{n=1}^\infty \bigsqcup_{r=0}^n \Herm_{n,r}^+$ consisting of all PSD matrices is not a smooth manifold in the usual topology, the differential geometric approach developed in aforementioned works is no longer applicable. As far as we are aware, there are essentially two methods to resolve the issue: one is based on the point-set distance  \cite{lim2019geometric} and the other is based on the optimal transport of probability distributions \cite{AJA22,CL22}. Although both methods are beneficial in some circumstances, they are not able to completely address the problem. For example, the point-set distance proposed in \cite{lim2019geometric} is only well-defined on $\bigcup_{n=1}^\infty \Herm_{n}^{++}$. The optimal transport based distance in \cite{AJA22,CL22} demonstrates some desirable theoretical properties but it is only explicitly computable in a few special cases.  

Besides those geometric measurements, there are also candidates in (quantum) information theory. Examples include Kullback-Leibler divergence \cite{SR51}, R\'{e}nyi divergence \cite{JC82}, S-divergence \cite{Sra12}, quantum relative entropy \cite{Wehrl78}, Umegaki relative entropy \cite{Hisaharu62}, Belavkin-Staszewski relative entropy \cite{BS82} and quantum Jensen-Shannon Divergence \cite{APD05}. However, these information theoretic measurements are not well-defined on $\bigcup_{n=1}^\infty \bigsqcup_{r=0}^n \Herm_{n,r}^+$. To be more specific, the first three in the preceding list are only well-defined for positive definite (PD) matrices, while the rest are only defined for a pair of PSD matrices $(A,B)$ where $\ker (B) \subseteq \ker(A)$.

In this paper, we present a unified framework (cf. Section~\ref{sec:GD}) to construct similarity measurements on the whole of $\bigcup_{n=1}^\infty \bigsqcup_{r=0}^n \Herm_{n,r}^+$. Ingredients that underlie our framework are the fiber bundle structure of $\Herm_{n,r}^+$ (cf. Section~\ref{sec:bundle}), a generalization of Hausdorff distance (cf. Subsubsection~\ref{subsub:Hausdorff}) and the idea of point-set distance \cite{ye2016schubert,lim2019geometric}. Given a function $d$ (resp. $\delta$) which measures the similarity between two non-equidimensional subspaces (resp. PD matrices), we denote by $\operatorname{GD}_{d,\delta}$ the measurement on $\bigcup_{n=1}^\infty \bigsqcup_{r=0}^n \Herm_{n,r}^+$ obtained by our framework (cf. Definition~\ref{def:GD}). By definition, $\operatorname{GD}_{d,\delta}$ extends both $d$ and $\delta$ (Proposition~\ref{prop:restriction}). We investigate both theoretical and computational aspects of $\operatorname{GD}_{d,\delta}$ in Sections~\ref{sec:geometry}--\ref{sec:distance}. 

On the theoretical side, we prove the following results:
\begin{enumerate}[(a)]
\item If $d$ and $\delta$ are induced from geodesic distances on Grassmann manifolds \cite{ye2016schubert} and cones of PD matrices \cite{lim2019geometric} respectively, then $\operatorname{GD}_{d,\delta}$ can be interpreted as the cost of a parallel transport (cf. Subsection~\ref{subsec:parallel transport}) and the length of a quasi-geodesic curve (cf. Subsection~\ref{subsec:quasi-geodesic}).
\item Most commonly-used divergences for PD matrices can be extended to non-equidimensional cases, by two types of point-set distances. More importantly, the two extended divergences coincide (cf. Subsections~\ref{subsec:geodist} and \ref{subsec:diverg}). Such extended divergences provide us more computable choices of $\delta$ for $\operatorname{GD}_{d,\delta}$.
\item There is a unique projection/lift map associated to an extended divergence (cf. Subsection~\ref{subsec:proj and lift}).
\end{enumerate}

On the computational side, there is an explicit formula for $\operatorname{GD}_{d,\delta}$ generically (cf. Section~\ref{sec:explicit formulae}). To be more precise, $A\in \Herm^+_{n,r}$ and $B\in \Herm^+_{m,s}$ ($r \le s$) satisfies $\ker(A) \cap \ker(B)^\perp = 0$, then $\operatorname{GD}_{d,\delta}(A,B)$ is easily computable (cf. Corollary~\ref{cor:simpify delta_H} and Tables~\ref{tab:explicit formulae 1}--\ref{tab:explicit formulae 4}). It is straightforward to verify that the subset $S\subseteq \Herm^+_{n,r} \times \Herm^+_{m,s}$ consisting of $(A,B)$ such that $\ker(A) \cap \ker(B)^\perp = 0$ is defined by some polynomial equations. Thus, formulae for $\operatorname{GD}_{d,\delta}$ listed in Tables~\ref{tab:explicit formulae 1}--\ref{tab:explicit formulae 4} are applicable to the full measure subset $(\Herm^+_{n,r} \times \Herm^+_{m,s}) \setminus S$, if we equip $\Herm^+_{n,r} \times \Herm^+_{m,s}$ with the Lebesgue measure.
\section{Preliminaries and notations}
\subsection{Positive semidefinite matrix}
Let $\mathcal{k}$ be either $\mathbb{R}$ or $\mathbb{C}$. Given $A\in \mathcal{k}^{n\times n}$ such that $A^\ast = A$, we say that $A$ is a \emph{positive semidefinite (PSD) matrix}, denoted by $A \succeq 0$,  if $x^\ast A x \ge 0$ for all $x\in \mathcal{k}^n$. We say that $A$ is a \emph{positive definite (PD) matrix}, denoted by $A \succ 0$, if  $A\succeq 0$ and $A$ is invertible. Let $r\le n$ be a positive integer. We denote 
\begin{align*}
\Herm_{n}&\coloneqq \lbrace 
A\in \mathcal{k}^{n\times n}: A^\ast = A 
\rbrace,\\
\Herm^{+}_{n,r} &\coloneqq \lbrace 
A\in \Herm_{n}: \rank(A) = r,~A \succeq 0 
\rbrace,\\
\Herm_{n}^{++} &\coloneqq \Herm^{+}_{n,n} = \lbrace
A\in \Herm_{n}: A \succ 0 
\rbrace.
\end{align*}
\subsection{Principal vectors and angles}
Let $\mathbb{A} \in \Gr(r,n)$ and $\mathbb{B}\in \Gr(s,n)$ be given subspaces. Let $A$ (resp. $B$) be an $n\times r$ (resp. $n\times s$) matrix whose column vectors consist an orthonormal basis of $\mathbb{A}$ (resp. $\mathbb{B}$). Let $A^\ast B = U \Sigma V^\ast$ be a full singular value decomposition (SVD) of $A^\ast B$, where $U\in U(r), V\in U(s)$ and $\Sigma$ is an $r\times s$ diagonal matrix whose diagonal elements are 
\[
\sigma_1 \ge \cdots \ge \sigma_{\min\{r,s\}} \ge 0.
\] 
We write 
\[
AU = \begin{bmatrix}
a_1 & \dots & a_r
\end{bmatrix}\in \mathcal{k}^{n\times r}, \quad BV = \begin{bmatrix}
b_1 & \dots & b_s
\end{bmatrix}\in \mathcal{k}^{n\times s}
\]
and call $\left( \{a_1,\dots, a_r\},\{b_1,\dots, b_s\} \right)$ a set of \emph{principal vectors} between $\mathbb{A}$ and $\mathbb{B}$. For $1 \le i \le \min\{r,s\}$, $\theta_i \coloneqq \arccos(a_i^\ast b_i)$ is called the \emph{$i$-th principal angle} between $\mathbb{A}$ and $\mathbb{B}$. It is clear from the definition that $\{a_1,\dots, a_r\}$ (resp. $\{b_1,\dots, b_s\}$) is an orthonormal basis of $\mathbb{A}$ (resp. $\mathbb{B}$), such that for each $1 \le i \le r$ and $1 \le j \le s$, 
\[
a_i^\ast b_j =\begin{cases}
\cos(\theta_i),&~\text{if}~i = j \le \min\{r,s\}, \\
0,&~\text{otherwise}.
\end{cases}
\]
\section{Bundle structures and stratifications}\label{sec:bundle}
Since each $A\in \Herm^{+}_{n,r}$ has rank $r$, we have $\dim \left( \ker(A)^{\perp} \right) = r$. Thus we may define a surjective map $\pi_{n,r}$ that follows.
\[
\pi_{n,r}: \Herm^{+}_{n,r} \to \Gr(r,n),\quad \pi_{n,r}(A) = \ker(A)^\perp.
\]
Given any $\mathbb{U} \in \Gr(r,n)$, the fiber $\pi_{n,r}^{-1}(\mathbb{U})$ consists of all $A\in \mathcal{k}^{n\times n}$ such that $\ker(A) = \mathbb{U}^{\perp}$ and $A|_{\mathbb{U}}: \mathbb{U} \to \mathbb{U}$ is a positive operator. As a consequence, $\pi_{n,r}^{-1}(\mathbb{U})$ can be identified with the cone of all positive linear operators on $\mathbb{U}$. In fact, we we have the following:
\begin{lemma}\label{lem:bundle}
$\pi_{n,r}: \Herm^{+}_{n,r} \to \Gr(r,n)$ is a fiber bundle on $\Gr(r,n)$ whose fiber is the open cone $\Herm_{r}^{++}$. Moreover, $\Herm^{+}_{n,r}$ can be identified with a sub-bundle of $\Herm (\mathcal{E}_r)$ where $\mathcal{E}_r \to \Gr(r,n)$ is the tautological  subbundle of $\Gr(r,n)$ and the fiber of $\Herm (\mathcal{E}_r)$ over $\mathbb{U}\in \Gr(r,n)$ is the space of self-adjoint operators on $\mathbb{U}$.
\end{lemma}
It is obviously true that any $A \in \Herm^{+}_{n,r}$ can be uniquely determined by the pair 
\[
(\ker(A)^\perp, A\mid_{\ker(A)^\perp})\in \Gr(r,n) \times \Herm^{++}_{r}.
\]
However, this does not imply that $\Herm^{+}_{n,r} \simeq \Gr(r,n) \times \Herm^{++}_{r}$ since $A\mid_{\ker(A)^\perp}$ acts on different $r$-dimensional subspaces of $\mathcal{k}^n$ as $A$ varies, if $r < n$.
Given integers $1 \le p, q \le n$ and $\max\{p-q,0\} \le l \le p$, we define the following incidence set: 
\begin{equation}\label{eq:E_l}
E_l \coloneqq \lbrace
(\mathbb{U}, \mathbb{V}) \in \Gr(p,n) \times \Gr(q,n): \dim \left( \mathbb{U} \cap \mathbb{V}^\perp \right) = l
\rbrace
\end{equation}
We denote by $\proj_1$ the projection map from $E_l$ onto its first factor, i.e., 
\[
\proj_1: E_l \to \Gr(p,n),\quad \proj_1(\mathbb{U},\mathbb{V}) = \mathbb{U}.
\]
\begin{lemma}\label{lem:stratification}
We have the following: 
\begin{enumerate}[(a)]
\item $\{E_l\}_{l = \max\{0,p-q\}}^p$ is a stratification of $\Gr(p,n) \times \Gr(q,n)$.\label{lem:stratification:item1}
\item $\proj_1: E_l \to \Gr(p,n)$ is a fiber bundle on $\Gr(p,n)$, whose fiber is the Schubert cell in $\Gr(q,n)$ associated with the sequence $(\underbrace{{q - p + l,\dots, q - p + l}}_{\text{$(n - q)$ times}})$. In particular, for each $\mathbb{U}\in \Gr(p,n)$, the fiber $\proj_1^{-1}(\mathbb{U})$ is an analytic submanifold of $\Gr(q,n)$, of dimension $q(n-q) - l(q-p+l)$. \label{lem:stratification:item2}  
\end{enumerate}
\end{lemma}
Given decreasing integers $m - k \ge a_1 \ge \dots \ge  a_k \ge 0$, the Schubert cell in $\Gr(k,m)$ associted with $(a_1,\dots, a_k)$ consists of all $\mathbb{W}\in \Gr(k,m)$ such that $\dim (\mathbb{W} \cap V_{m-k + i -a_i}) = i$ for all $i \ge 1$. Here $\mathbb{V}_1 \subsetneq \cdots \subsetneq \mathbb{V}_m = \mathcal{k}^m$ is a fixed complete flag in $\mathcal{k}^m$.
\begin{proof}
The proof of \eqref{lem:stratification:item1} is straightforward from the definition of $E_l$ while \eqref{lem:stratification:item2} is a consequence of \cite[Theorem~5.1]{wong1970class}.
\end{proof}
Let $r \le s \le n$ be positive integers. According to Lemmas~\ref{lem:bundle} and \ref{lem:stratification}, we obtain a stratification of $\Herm^{+}_{n,r} \times \Herm^{+}_{n,s}$: 
\[
K_l \coloneqq ( \pi_{n,r} \times \pi_{n,s})^{-1} (E_l),\quad 0 \le l \le r. 
\]
\begin{proposition}\label{prop:bundle}
Let $r \le s \le n$ be positive integers. For each $0 \le l \le r$, $K_l$ is a non-trivial fiber bundle on $E_l$ whose fiber is the cone $\Herm^{++}_r \times \Herm^{++}_s$.
\end{proposition}
\begin{proof}
It is sufficient to prove the non-triviality of $K_l$. We proceed by contradiction. Assume that $\pi_{n,r} \times \pi_{n,s}: K_l  \to E_l$ is a trivial bundle. This implies that it has $\binom{r+1}{2} + \binom{s+1}{2}$ everywhere linearly independent global sections, since $\Herm^{+}_m$ is an open cone of full dimension in $\Herm_m$ for each positive integer $m$. Moreover, these sections are everywhere linearly independent global sections of $
\Herm (\mathcal{E}_r) \oplus \Herm (\mathcal{E}_s)$ over $E_l$.

Since $\proj_1: E_l \to \Gr(r,n)$ is a fiber bundle with contractible fiber, it has a section $\sigma: \Gr(r,n) \to E_l$. We observe that $\sigma^\ast \left( \Herm (\mathcal{E}_r) \oplus \Herm (\mathcal{E}_s) \right) = \Herm (\mathcal{E}_r)\oplus E'$ where $E'$ is some vector bundle of rank $\binom{s+1}{2}$ on $\Gr(r,n)$. Thus the pull-back of global sections of $K_l$ over $E_l$ supplies $\binom{r+1}{2}$ everywhere linearly independent global sections of $\Herm (\mathcal{E}_r)$ over $\Gr(r,n)$, which contradicts to the fact that the vector bundle $\Herm (\mathcal{E}_r)$ is non-trivial.
\end{proof}
For clarity, we summarize consequences of Lemmas~\ref{lem:bundle}, \ref{lem:stratification} and Proposition~\ref{prop:bundle} in Figure~\ref{fig:bundle}.
\begin{figure}[!htbp]
  \centering
  \begin{tikzcd}
\Herm_r^{++} \times \Herm_s^{++} \arrow[r,no head] & K_l \arrow[r,symbol=\subseteq] \arrow{d}{\pi_{n,r} \times \pi_{n,s}}   & \Herm_{n,r}^{+} \times \Herm_{n,s}^{+} \arrow{d}{\pi_{n,r} \times \pi_{n,s}}     \\
\sigma_{q-p+l,\dots, q-p+l} \arrow[r,no head] & E_l \arrow{d}{\proj_1}  \arrow[r,symbol=\subseteq] & \Gr(r,n) \times \Gr(s,n) \arrow{dl}{\proj_1}    \\
& \Gr(r,n)  &  
  \end{tikzcd}
  \caption{Bundle structures and stratifications}
  \label{fig:bundle}
\end{figure}
\section{Geometric distance between PSD matrices}\label{sec:GD}
In this section, we propose a geometric method to measure the difference between $A\in \Herm_{n,r}^+$ and $B\in \Herm_{m,s}^+$, where possibly $n \ne m$ and $r \ne s$. 
\subsection{The construction of geometric distance}\label{subsec:construction of GD}
The geometric distance to be defined in Definition~\ref{def:GD} is based on the bundle structures and stratifications described in Figure~\ref{fig:bundle}. Without loss of generality, we may assume that $n \ge m$ so that we have the natural inclusion $\mathcal{k}^m \subseteq \mathcal{k}^n$. In particular, it induces an embedding 
\[
\iota_{m,n}: \Gr(s,m) \hookrightarrow \Gr(s,n),\quad \operatorname{span}\{v_1,\dots, v_s\} \mapsto \operatorname{span}\{\widetilde{v_1},\dots, \widetilde{v_s}\},
\]
where $\widetilde{v} = \begin{bmatrix}
v \\
0
\end{bmatrix} \in \mathcal{k}^{n}$ for each column vector $v\in \mathcal{k}^m$. Moreover, there is an induced inclusion
\begin{equation}\label{eq:iota star}
\iota_{m,n}^\ast: \Herm_{m,s}^+ \hookrightarrow \Herm_{n,s}^+,\quad A \mapsto \begin{bmatrix}
A & 0 \\
0 & 0
\end{bmatrix}.
\end{equation}
Thus we have the diagram in Figure~\ref{fig:bundle1}, where $L_l  \coloneqq \left(\operatorname{Id} \times \iota_{m,n}^{\ast}\right)^{-1} (K_l)$ and $F_l \coloneqq \left(\operatorname{Id} \times \iota_{m,n}\right)^{-1}(E_l)$ for each $\max\{r-s,0\} \le l \le r$.
\begin{figure}[!htbp]
  \centering
  \begin{tikzcd}
L_l \arrow[r,symbol=\subseteq] \arrow{d}{\pi_{n,r} \times \pi_{m,s}}  & \Herm_{n,r}^{+} \times \Herm_{m,s}^{+} \arrow{d}{\pi_{n,r} \times \pi_{m,s}} \arrow{r}{\operatorname{Id} \times \iota^\ast_{m,n}}   & \Herm_{n,r}^{+} \times \Herm_{n,s}^{+}  \arrow{d}{\pi_{n,r} \times \pi_{n,s}} \arrow[r,symbol=\supseteq]  &     K_l \arrow{d}{\pi_{n,r} \times \pi_{n,s}}  \\
F_l \arrow[r,symbol=\subseteq]  & \Gr(r,n) \times \Gr(s,m)  \arrow{r}{\operatorname{Id} \times \iota_{m,n}}    & \Gr(r,n) \times \Gr(s,n) \arrow[r,symbol=\supseteq]   & E_l
  \end{tikzcd}
 \caption{Induced inclusions}
  \label{fig:bundle1}
\end{figure}
\subsubsection{A generalization of Hausdorff distance}\label{subsub:Hausdorff} To begin with, we slightly generalize the Hausdorff distance.
Let $\mathcal{X},\mathcal{Y}$ be two sets and let $f:\mathcal{X}\times \mathcal{Y} \to \mathbb{R}_+$ be a function. For $x\in \mathcal{X}, \mathcal{Y}_1 \subseteq \mathcal{Y}, y\in \mathcal{Y}$ and $\mathcal{X}_1 \subseteq \mathcal{X}$, we define 
\[
f(x,\mathcal{Y}_1) \coloneqq \inf_{y_1\in \mathcal{Y}_1} f(x,y_1),\quad f(\mathcal{X}_1,y) \coloneqq \inf_{x_1\in \mathcal{X}_1} f(x_1,y).
\]
Given a subset $\mathcal{Z} \subseteq \mathcal{X} \times \mathcal{Y}$, we denote by $p_1: \mathcal{Z} \to \mathcal{X}$ (resp. $p_2: \mathcal{Z} \to \mathcal{Y}$) the projection of $\mathcal{Z}$ onto its first (resp. second) factor. On the power set $2^{\mathcal{X} \times \mathcal{Y}}$ of $\mathcal{X} \times \mathcal{Y}$, we consider the function
\begin{equation}\label{eq:Hausdorff}
f_H: 2^{\mathcal{X} \times \mathcal{Y}} \to \mathbb{R}_+,\quad \mathcal{Z} \mapsto \max \left\lbrace  \sup_{x\in p_1(\mathcal{Z})} f(x,p_2(p_1^{-1}(x))),
\sup_{y\in p_2(\mathcal{Z})} f(p_1(p_2^{-1}(y)),y) \right\rbrace.
\end{equation}
Clearly, $f_H$ has the following properties: 
\begin{itemize}
\item If $\mathcal{X} = \mathcal{Y}$ is a metric space with metric $f = d$, then $f_H(\mathcal{A} \times \mathcal{B})$ is the Hausdorff distance between $\mathcal{A}$ and $\mathcal{B}$ for any subsets $\mathcal{A}, \mathcal{B}\subseteq \mathcal{X}$.
\item If $\mathcal{Z} = \{(x,y)\}$ then $f_H(\mathcal{Z}) = f(x,y)$.
\item If $\mathcal{X}$ and $\mathcal{Y}$ are compact topological spaces and $f$ is continuous, then $f_H(\mathcal{Z}) < \infty$ for any $\mathcal{Z} \subseteq \mathcal{X} \times \mathcal{Y}$.
\end{itemize}
\subsubsection{Geometric distance}
Given positive operators $X$ on $\mathbb{U} \in \Gr(r,n)$, $Y$ on $\mathbb{V}\in \Gr(s,n)$, we let $(\mathbf{u},\mathbf{v}) \coloneqq \left( \{u_1,\dots, u_r\}, \{v_1,\dots, v_s\} \right)$ be a set of principal vectors between $\mathbb{U}$ and $\mathbb{V}$. Since $\mathbf{u}$ (resp. $\mathbf{v}$) is an orthonormal basis of $\mathbb{U}$ (resp. $\mathbb{V}$) and $X$ (resp. $Y$) is a positive operator on $\mathbb{U}$ (resp. $\mathbb{V}$), $X$ (resp. $Y$) is represented by a matrix $M_X(\mathbf{u},\mathbf{v})$ (resp. $M_Y(\mathbf{u},\mathbf{v})$) in $\Herm_r^{++}$ (resp. $\Herm_s^{++}$) with respect to $\mathbf{u}$ (resp. $\mathbf{v}$). Indeed, we have 
\begin{equation}\label{eq:matrix representation}
M_X(\mathbf{u},\mathbf{v}) = ( u_i^{\ast} X(u_j))_{i,j=1}^r \in \Herm_r^{++}, \quad
M_Y(\mathbf{u},\mathbf{v}) = ( v_p^{\ast} Y (v_q))_{p,q=1}^s \in \Herm_s^{++}. 
\end{equation}
We denote by $\mathcal{Z}_{X,Y} $ the subset of $\Herm_r^{++} \times \Herm_s^{++}$ consisting of matrix representations of $(X,Y)$ with respect to all possible principal vectors $(\mathbf{u},\mathbf{v})$ between $\mathbb{U}$ and $\mathbb{V}$:
\begin{equation}\label{eq:ZXY}
\mathcal{Z}_{X,Y} \coloneqq \left\lbrace
(M_{X}(\mathbf{u},\mathbf{v}),M_{Y}(\mathbf{u},\mathbf{v})): (\mathbf{u},\mathbf{v})~\text{are principal vectors between $\mathbb{U}$ and $\mathbb{V}$}
\right\rbrace.
\end{equation}
\begin{definition}[Geometric distance]\label{def:GD}
For any function $d: \Gr(r,n) \times \Gr(s,n) \to \mathbb{R}_+$ and $\delta: \Herm_r^{++} \times \Herm_s^{++} \to \mathbb{R}_+$, the \emph{geometric distance} between $A\in \Herm_{n,r}^+$ and $B\in \Herm_{m,s}^+$ is defined by
\begin{equation}\label{eq:GD}
\operatorname{GD}_{d,\delta}(A,B) \coloneqq \left( d^2(\pi_{n,r}(A),\iota_{m,n} \circ \pi_{m,s}(B)) + \delta_H^2 (\mathcal{Z}_{A_{\mid \pi_{n,r}(A)}, \iota_{m,n}^\ast(B)_{\mid \iota_{m,n} \circ \pi_{m,s}(B)}}) \right)^{\frac{1}{2}}.
\end{equation}
Here $\mathcal{Z}_{A_{\mid \pi_{n,r}(A)}, \iota_{m,n}^\ast(B)_{\mid \iota_{m,n} \circ \pi_{m,s}(B)}}$ is the set defined for $A_{\mid \pi_{n,r}(A)}$ and $\iota_{m,n}^\ast(B)|_{\iota_{m,n} \circ \pi_{m,s}(B)}$ in \eqref{eq:ZXY} and $\delta_H$ is the generalized Hausdorff function defined in \eqref{eq:Hausdorff}.
\end{definition}
\begin{remark}
We have several remarks in order. 
\begin{enumerate}
\item One can choose the function $d$ (resp. $\delta$) to be any measurement to distinguish non-equidimensional  subspaces of $\mathbb{R}^n$ (resp. PD matrices). In particular, if $r = s$, then $d$ and $\delta$ can be respectively chosen to be distance functions on $\Gr(r,n)$ and $\Herm_r^{++}$.
\item We recall from Figure~\ref{fig:bundle} that $\pi_{n,r} \times \pi_{n,s}: \Herm_{n,r}^+ \times \Herm_{n,s}^+ \to \Gr(r,n) \times \Gr(s,n)$ is a fiber bundle, thus for any $(\mathbb{U}, \mathbb{V}) \in \Gr(r,n) \times \Gr(s,n)$ there exists some neighbourhood $\mathcal{U} \subseteq \Gr(r,n)$ (resp. $\mathcal{V} \subseteq \Gr(s,n)$) of $\mathbb{U}$ (resp. $\mathbb{V}$) such that 
\begin{equation}\label{eq:local trivilization}
\pi_{n,r}^{-1}(\mathcal{U}) \times \pi_{n,s}^{-1}(\mathcal{V}) \simeq \left( \mathcal{U} \times \mathcal{V} \right) \times \left( \Herm_{r}^{++} \times \Herm_{s}^{++} \right).
\end{equation}
Therefore, on $\pi_{n,r}^{-1}(\mathcal{U}) \times \pi_{n,s}^{-1}(\mathcal{V})$, $\operatorname{GD}_{d,\delta}$ can be recognized as the product ``metric" on $\left( \mathcal{U} \times \mathcal{V} \right) \times \left( \Herm_{r}^{++} \times \Herm_{s}^{++} \right)$ induced by $d$ and $\delta$. In particular, $\operatorname{GD}_{d,\delta}$ is indeed the product metric on $\pi_{n,r}^{-1}(\mathcal{U}) \times \pi_{n,s}^{-1}(\mathcal{V})$ if $r = s$ and both $d,\delta$ are distance functions.
\item Unfortunately, we are not able to define $\operatorname{GD}_{d,\delta}$ as the product ``metric" on the whole $\Herm_{n,r}^+ \times \Herm_{n,s}^+$ since the bundle $\pi_{n,r} \times \pi_{n,s}: \Herm_{n,r}^+ \times \Herm_{n,s}^+ \to \Gr(r,n) \times \Gr(s,n)$ is non-trivial. To be more precise, we notice that the local trivialization in \eqref{eq:local trivilization} depends on a choice of principal vectors between $\mathbb{U}$ and $\mathbb{V}$. Thus ambiguities arise if $(\mathbb{U},\mathbb{V}) \in E_l$ for $l \ge 1$. Here $E_l$ is the subset of $\Gr(r,n) \times \Gr(s,n)$ defined in Figure~\ref{fig:bundle}. The $\delta_H$ term in \eqref{eq:GD} is introduced to resolve such ambiguities.
\end{enumerate} 
\end{remark}
\begin{example}\label{ex:GD}
We consider the case where $\mathcal{k} =\mathbb{R}, r = s, m = n$ and 
\begin{align*}
&d: \Gr(r,n) \times \Gr(r,n) \to \mathbb{R}_+, \quad (\mathbb{U},\mathbb{V}) \mapsto  \left( \sum_{i=1}^r \theta_i^2 \right)^{1/2},\\
&\delta: \Herm_r^{++} \times \Herm_r^{++} \to \mathbb{R}_+, \quad (C,D) \mapsto \left( \sum_{i=1}^r \log^2 \lambda_i\right)^{1/2}.
\end{align*}
Here $\theta_i$ is the $i$-th principal angle between $\mathbb{U}$ and $\mathbb{V}$ and $\lambda_i > 0$ is the $i$-th largest eigenvalue of $C^{-1/2} D  C^{-1/2}$, $1\le i \le r$. Let 
\[
A=\begin{bmatrix} 1 & 0 & 0 & 0 & 0\\
                       0 & 1 & 0 & 0 & 0\\
                       0 & 0 & 1/2& 0 & 0\\
                       0 & 0 & 0 & 0 & 0\\
                       0 & 0 & 0 & 0 & 0
       \end{bmatrix},\quad B=\begin{bmatrix} 1 & 0 & 0 & 0 & 0\\
                       0 & 0 & 0 & 0 & 0\\
                       0 & 0 & 0 & 0 & 0\\
                       0 & 0 & 0 & 1 & 0\\
                       0 & 0 & 0 & 0 & 2
       \end{bmatrix}.
\]
Then $\mathbb{U} = \pi_{5,3}(A) = \operatorname{span}\{e_1,e_2,e_3\}$, $\mathbb{V} = \pi_{5,3}(B) = \operatorname{span}\{e_1,e_4,e_5\}$ where $e_k$ is $k$-th standard basis vector of $\mathbb{R}^5, 1\le k \le 5$. Principal vectors between $\mathbb{U}$ and $\mathbb{V}$ are of the form
\begin{align*}
\mathbf{u} &= \{ \varepsilon e_1,  \cos(\psi) e_2 - \sin(\psi) e_3, \sin(\psi) e_2 + \cos(\psi) e_3\}, \\
\mathbf{v} &=\{  \varepsilon e_1, \cos(\theta)  e_4 - \sin(\theta)  e_5,\sin(\theta)   e_4 + \cos(\theta)  e_5\},
\end{align*}
where $\psi, \theta \in [0,2\pi)$ and $\varepsilon\in \{-1,1\}$. Thus elements $(M_A(\mathbf{u},\mathbf{v}),M_B(\mathbf{u},\mathbf{v}))$ in $\mathcal{Z}_{A_{|_{\mathbb{A}}}, B_{|_{\mathbb{B}}}}$ can be written as
\begin{equation}\label{ex:GD1}
\left(  \begin{bmatrix}
                   1 & 0 & 0\\
                   0 & \frac{1+\cos^2\psi}{2} & \frac{\sin\psi\cos\psi}{2} \\
                   0 & \frac{\sin\psi\cos\psi}{2} & \frac{1+\sin^2\psi}{2}
               \end{bmatrix}, 
               \begin{bmatrix}
                   1 & 0 & 0\\
                   0 & 1+\sin^2\theta & -\sin\theta\cos\theta \\
                   0 & -\sin\theta\cos\theta& 1+\cos^2\theta
               \end{bmatrix}
               \right),\quad \psi, \theta \in [0,2\pi).
\end{equation}
Clearly we have 
\begin{equation}\label{ex:GD2}
\delta^2(M_A(\mathbf{u},\mathbf{v}),M_B(\mathbf{u},\mathbf{v})) = 2(\log \lambda - 1)^2 + 2,
\end{equation}
where $\lambda$ is a root of $t^2 - (4 + \cos^2 (\theta - \psi) ) t + 4 = 0$. This implies 
\[
\delta_H(\mathcal{Z}_{A_{|_{\mathbb{A}}}, B_{|_{\mathbb{B}}}})^2 =\min_{\substack{\alpha \in [0,2 \pi),\\\lambda^2 - (4 +\cos^2 \alpha) \lambda + 4 = 0}} 2\left[ (\log\lambda - 1)^2 + 1 \right] = 2,
\]
thus $\operatorname{GD}_{d,\delta}(A,B) =\sqrt{\frac{\pi^2}{4} + 2}$.
\end{example}
\subsection{Basic properties} Having defined the geometric distance in Subsection~\ref{subsec:construction of GD}, we discuss some of its basic properties in this subsection. 
\begin{lemma}\label{lem:Z_XY}
Assume that $X$ and $Y$ are positive operators on $\mathbb{U} \in \Gr(r,n)$ and $\mathbb{V}\in \Gr(s,n)$ respectively. Let $(\mathbf{u},\mathbf{v}) = ( \{u_1,\dots, u_r\},\{v_1,\dots, v_s\})$ and  $(\mathbf{u}',\mathbf{v}') = ( \{u'_1,\dots, u'_r\},\{v'_1,\dots, v'_s\})$ be two sets of principal vectors between $\mathbb{U}$ and $\mathbb{V}$. If $(\mathbb{U},\mathbb{V})\in E_l$ and $r \le s$, then there exist $S \in U(l), T\in U(s - r + l)$ and a block diagonal $P\in U(r - l)$ such that 
\[
\begin{bmatrix}
u_1 & \dots & u_r
\end{bmatrix} \begin{bmatrix}
P & 0\\
0 & S
\end{bmatrix}  = \begin{bmatrix}
u'_1 & \dots & u'_r
\end{bmatrix},\quad 
\begin{bmatrix}
v_1 & \dots & v_s
\end{bmatrix}\begin{bmatrix}
P & 0 \\
0 & T
\end{bmatrix}  = \begin{bmatrix}
v'_1 & \dots & v'_s
\end{bmatrix}.
\]
\end{lemma}
\begin{proof}
Since both $(\mathbf{u},\mathbf{v})$ and $(\mathbf{u}',\mathbf{v}')$ are sets of principal vectors between $\mathbb{U}$ and $\mathbb{V}$, we have 
\[
\begin{bmatrix}
u^\ast_1 \\
\vdots \\ 
u^\ast_r
\end{bmatrix} \begin{bmatrix}
v_1 & \dots & v_s
\end{bmatrix} = \begin{bmatrix}
\sigma_1 & 0 & \dots & 0 & 0 & \dots & 0 \\
0 & \sigma_2 & \dots & 0 & 0 & \dots & 0 \\
\vdots & \vdots & \ddots & \vdots& \vdots & \ddots & \vdots \\ 
0 & 0 & \dots & \sigma_r & 0 & \dots & 0
\end{bmatrix} = \begin{bmatrix}
{u'}^\ast_1 \\
\vdots \\ 
{u'}^\ast_r
\end{bmatrix} \begin{bmatrix}
v'_1 & \dots & v'_s
\end{bmatrix} .
\]
Here $\arccos (\sigma_i)$ is the $i$-th principal angle between $\mathbb{U}$ and $\mathbb{V}$, $1 \le i \le r$. If $(\mathbb{U},\mathbb{V})\in E_l$, then $\sigma_{r-l+1} = \dots =\sigma_r = 0$. This indicates the existence of $S \in U(l)$ and $T\in U(s - r + l)$. Moreover, If 
\[
\sigma_1 = \dots = \sigma_{i_1} > 
\sigma_{i_1 + 1} = \dots = \sigma_{i_2} > \dots 
> \sigma_{i_{q-1} + 1} = \dots = \sigma_{r-l},
\]
then we can find $P_1\in U(i_1), P_2\in U(i_2 - i_1), \dots, P_{q} \in U(r - l - i_{q-1})$ such that 
\[
P  \coloneqq \operatorname{diag}(P_1 ,\dots, P_q) \in U(r - l)
\]
has the desired property.
\end{proof}
Let $\delta: \Herm_r^{++} \times \Herm_s^{++} \to \mathbb{R}_+$ be a function. We say that $\delta$ is \emph{$U(r)$-invariant} if $r \le s$ and 
\[
\delta \left( P C P^\ast, \begin{bmatrix}
P & 0 \\
0 & I_{s-r}
\end{bmatrix} D \begin{bmatrix}
P^\ast & 0 \\
0 & I_{s-r}
\end{bmatrix} \right) = \delta (C,D),
\]
for any $P\in U(r), C \in \Herm_r^+$ and $D\in \Herm_s^+$. 
\begin{proposition}\label{prop:simpify delta_H}
Let $\delta: \Herm_r^{++} \times \Herm_s^{++} \to \mathbb{R}_+$ be $U(r)$-invariant. Suppose that $X$ is a positive operator on $\mathbb{U}\in \Gr(r,n)$ and $Y$ is a positive operator on $\mathbb{V}\in \Gr(s,n)$ such that $(\mathbb{U},\mathbb{V}) \in E_l$. Then for any set of principal vectors $(\mathbf{u},\mathbf{v})$ between $\mathbb{U}$ and $\mathbb{V}$, we have  
\begin{equation}\label{prop:simpify delta_H:eq}
\delta_H(\mathcal{Z}_{X,Y})  = 
\max_{Y_1\in \mathcal{Y}_1} \delta(M_X(\mathbf{u},\mathbf{v}) ,Y_1),
\end{equation}
where
\begin{equation*}
\mathcal{Y}_1 \coloneqq \left\lbrace
\begin{bmatrix}
I_{r - l} & 0 \\
0 & T
\end{bmatrix} M_Y(\mathbf{u},\mathbf{v}) \begin{bmatrix}
I_{r - l} & 0 \\
0 & T^\ast
\end{bmatrix}, T\in U(s - r + l)
\right\rbrace.
\end{equation*}
\end{proposition}
\begin{proof}
According to \eqref{eq:ZXY}, an element in $\mathcal{Z}_{X,Y}$ can be written as $(M_X(\mathbf{u}',\mathbf{v}'),M_Y(\mathbf{u}',\mathbf{v}'))$ for some set of principal vectors between $\mathbb{U}$ and $\mathbb{V}$. By Lemma~\ref{lem:Z_XY} and \eqref{eq:matrix representation}, there exist $P_0\in U(r-l), S_0\in U(l)$ and $T_0\in U(s - r + l)$ such that
\begin{align*}
M_X(\mathbf{u}',\mathbf{v}') &= 
\begin{bmatrix}
P_0 & 0 \\
0 & S_0
\end{bmatrix} M_X(\mathbf{u},\mathbf{v})\begin{bmatrix}
P_0^\ast & 0 \\
0 & S_0^\ast
\end{bmatrix},  \\
M_Y(\mathbf{u}',\mathbf{v}') &= 
\begin{bmatrix}
P_0 & 0 \\
0 & T_0
\end{bmatrix} M_Y(\mathbf{u},\mathbf{v})\begin{bmatrix}
P_0^\ast & 0 \\
0 & T_0^\ast
\end{bmatrix}.
\end{align*}
Since $\delta$ is $U(r)$-invariant, we must have
\[
\delta \left( M_X(\mathbf{u}',\mathbf{v}'),  M_Y(\mathbf{u}',\mathbf{v}') \right) = \delta \left( 
M_X(\mathbf{u},\mathbf{v}), 
\begin{bmatrix}
I_{r-l} & 0 \\
0 & T
\end{bmatrix} M_Y(\mathbf{u},\mathbf{v})\begin{bmatrix}
I_{r-l} & 0 \\
0 & T^\ast
\end{bmatrix}
\right),
\]
for some $T \in U(s - r+ l)$

Unravelling the definition, we obtain 
\begin{equation*}
\delta_H(\mathcal{Z}_{X,Y}) = \delta_H(\mathcal{X}_1 \times \mathcal{Y}_1) = \max\left\lbrace \max_{X_1\in \mathcal{X}_1}\min_{Y_1 \in \mathcal{Y}_1} \delta(X_1,Y_1), 
\max_{Y_1\in \mathcal{Y}_1}\min_{X_1 \in \mathcal{X}_1} \delta(X_1,Y_1)
\right\rbrace,
\end{equation*}
where $\mathcal{X}_1$ is the singleton consisting of $M_X(\mathbf{u},\mathbf{v})$ and 
\begin{equation*}
\mathcal{Y}_1 \coloneqq \left\lbrace
\begin{bmatrix}
I_{r - l} & 0 \\
0 & T
\end{bmatrix} M_Y(\mathbf{u},\mathbf{v}) \begin{bmatrix}
I_{r - l} & 0 \\
0 & T^\ast
\end{bmatrix}, T\in U(s - r + l)
\right\rbrace,
\end{equation*}
from which \eqref{prop:simpify delta_H:eq} follows easily.  
\end{proof}
We recall from Figure~\ref{fig:bundle1} that $L_0$ is a subset of $\Herm_{n,r}^+ \times \Herm_{m,s}^+$ consisting of $(A,B)$ such that principal angles between $\pi_{n,r}(A)$ and $\iota_{m,n} \circ \pi_{m,s}(B)$ are strictly smaller than $\pi/2$. As a consequence of Proposition~\ref{prop:simpify delta_H}, the geometric distance $\operatorname{GD}_{d,\delta}(A,B)$ can be computed without solving optimization problems in \eqref{prop:simpify delta_H:eq}, if $(A,B)\in L_0$. This is the content of Corollary~\ref{cor:simpify delta_H}. Since $L_0$ is in fact an open dense subset of $\Herm_{n,r}^+ \times \Herm_{m,s}^+$, $\operatorname{GD}_{d,\delta}(A,B)$ is actually easy to compute for generic $(A,B) \in \Herm_{n,r}^+ \times \Herm_{m,s}^+$.
\begin{corollary}\label{cor:simpify delta_H}
Let $\delta: \Herm_r^{++} \times \Herm_s^{++} \to \mathbb{R}_+$ be $U(r)$-invariant. Suppose that $\delta$ further satisfies 
\[
\delta\left(C, \begin{bmatrix}
I_r & 0 \\
0 & Q
\end{bmatrix} D \begin{bmatrix}
I_r & 0 \\
0 & Q^\ast
\end{bmatrix}  \right) = \delta(C,D) 
\]
for each $Q\in U(s-r), C \in \Herm_r^{++}, D \in \Herm_s^{++}$. Let $d: \Gr(r,n) \times \Gr(s,n) \to \mathbb{R}_+$ be any function. Then for any $(A,B)\in L_0$, we have
\begin{equation}
\operatorname{GD}_{d,\delta} (A,B) =\left( d^2(\pi_{n,r}(A), \iota_{m,n}\circ \pi_{m,s} (B)) + \delta^2(M_{A|_{\pi_{n,r}(A)}}(\mathbf{u},\mathbf{v}), M_{B|_{\iota_{m,n}\circ \pi_{m,s}(B)}}(\mathbf{u},\mathbf{v})) \right)^{\frac{1}{2}}.
\end{equation}
Here $(\mathbf{u},\mathbf{v})$ is any set of principal vectors between $\pi_{n,r}(A)$ and $\iota_{m,n}\circ \pi_{m,s} (B)$.
\end{corollary}
Clearly, $\operatorname{GD}_{d,\delta}$ is determined by $d$ and $\delta$. It is also true that one can recover $d$ and $\delta$ from $\operatorname{GD}_{d,\delta}$, under some mild conditions.
\begin{proposition}\label{prop:restriction}
Let $d,\delta$ and $\operatorname{GD}_{d,\delta}$ be as in Definition~\ref{def:GD}. We assume $r \le s$ and $d(\mathbb{X},\mathbb{Y}) = 0 = \delta (X,Y)$ whenever $\mathbb{U} \subseteq \mathbb{V}$ and $X$ is the upper left $r\times r$ submatrix of $Y$.
\begin{enumerate}[(a)]
\item\label{prop:restriction:item1} For $\mathbb{U}\in \Gr(r,n)$ (resp. $\mathbb{V}\in \Gr(s,m)$), we denote by $A_1$ (resp. $B_1$) the matrix in $\Herm^+_{n,r}$ (resp. $\Herm^+_{m,s}$) such that $A_1|_{\mathbb{U}} = \operatorname{Id}_{\mathbb{U}}$ (resp. $B_1|_{\mathbb{V}} = \operatorname{Id}_{\mathbb{V}}$) and $A_1|_{\mathbb{U}^\perp} = 0$ (resp. $B_1|_{\mathbb{V}^\perp} = 0$). Then we have 
\[
d(\mathbb{U}, \iota_{m,n} (\mathbb{V}) ) = \operatorname{GD}_{d,\delta}(A_1,B_1).
\]
\item\label{prop:restriction:item2} For $X \in \Herm_r^{++}$ (resp.$Y\in \Herm_s^{++}$ ), we define $A_2 = \begin{bmatrix}
X & 0 \\
0 & 0
\end{bmatrix} \in  \Herm_{n,r}^{+}$ (resp.$B_2 = \begin{bmatrix}
Y & 0 \\
0 & 0
\end{bmatrix} \in  \Herm_{m,s}^{+}$ ). If moreover $\delta$ satisfies conditions in Corollary~\ref{cor:simpify delta_H}, then
\[
\delta(X,Y) = \operatorname{GD}_{d,\delta}(A_2,B_2).
\] 
\end{enumerate}
\end{proposition}
\begin{proof}
We first prove \eqref{prop:restriction:item1}. By construction, it is clear that 
\[
\pi_{n,r} (A_1) = \mathbb{U},\quad \pi_{m,s} (B_1)  = \mathbb{V}.
\]
Moreover, since ${A_1}|_{\mathbb{U}} = \operatorname{Id}_\mathbb{U}$ and ${B_1}|_{\mathbb{V}} = \operatorname{Id}_\mathbb{V}$, we obtain that\[
\mathcal{Z}_{{A_1}_{\mid \pi_{n,r}(A)}, {\iota_{m,n}^\ast(B_1)}_{\mid \iota_{m,n} \circ \pi_{m,s}(B)}} = \{ (I_r, I_s) \}
\]
and $d(\mathbb{U}, \iota_{m,n} (\mathbb{V}) ) = \operatorname{GD}_{d,\delta}(A_1,B_1)$ follows immediately. 

To prove \eqref{prop:restriction:item2}, we simply notice that 
\[
\pi_{n,r}(A_2) = \mathbb{R}^r, \quad \pi_{m,s}(B_2) = \mathbb{R}^s.
\]
We also have 
\[
\mathcal{Z}_{{A_2}|_{\mathbb{R}^r}, {\iota_{m,n}^\ast(B_2)}|_{ \iota_{m,n}(\mathbb{R}^s) }} = \left\lbrace \left( \begin{bmatrix}
P^\ast & 0 \\
0 & Q^\ast
\end{bmatrix} X \begin{bmatrix}
P & 0 \\
0 & Q
\end{bmatrix}, P^\ast Y P\right): P^\ast P = I_s, Q^\ast Q = I_{r-s} \right\rbrace.
\]
Therefore, Corollary~\ref{cor:simpify delta_H} implies that 
\[
\operatorname{GD}_{d,\delta} (A_2,B_2) =\left( d^2(\mathbb{R}^r, \iota_{m,n} (\mathbb{R}^s)) + \delta^2(X, Y) \right)^{\frac{1}{2}} = \delta(X,Y).
\]
\end{proof}
Conditions on $d$ and $\delta$ in Proposition~\ref{prop:restriction} are satisfied by many common choices of $d$ and $\delta$. In fact, Proposition~\ref{prop:restriction} applies to all geometric distances appeared in Tables~\ref{tab:explicit formulae 1}--\ref{tab:explicit formulae 4}. Moreover, in the language of fiber bundle, $A_1$ in Proposition~\ref{prop:restriction}--\eqref{prop:restriction:item1} corresponds to the section $s_0:\Herm_{n,r}^{+} \to \Gr(r,n)$ of the fiber bundle $\pi_{n,r}: \Herm_{n,r}^{+} \to \Gr(r,n)$ defined by $s_0(\mathbb{U}) = \operatorname{Id}_{\mathbb{U}}$, and ditto for $B_1$. On the other side, $A_2$ in Proposition~\ref{prop:restriction}--\eqref{prop:restriction:item2} is simply the element in the fiber $\pi_{n,r}^{-1} (\mathbb{R}^r) \simeq \Herm^{++}_r$ that corresponds to $X$, and ditto for $B_2$.

\section{Differential geometry of \texorpdfstring{$\operatorname{GD}_{d,\delta}$}{}}\label{sec:geometry}
In this section, we discuss $\operatorname{GD}_{d,\delta}$ from the perspective of differential geometry. We concentrate on the case where $r = s$ and $m = n$ so that $\operatorname{GD}_{d,\delta}$ defines a function on $\Herm_{n,r}^+ \times \Herm_{n,r}^+$. We want to relate $\operatorname{GD}_{d,\delta}$ with the geometry of the fiber bundle $\pi_{n,r}: \Herm_{n,r}^+ \to \Gr(r,n)$. 
\subsection{Interpretation of \texorpdfstring{$\operatorname{GD}_{d,\delta}$}{} by parallel transport}\label{subsec:parallel transport}
Given $A\in \Herm_{n,r}^+$, we may decompose $A$ as $A = Q \begin{bmatrix}
M_A & 0 \\
0 & 0
\end{bmatrix} Q^\ast$ from some $Q\in U(n)$ and  $M_A\in \Herm^+_r$. We define 
\[
V_A \coloneqq  \left\lbrace Q \begin{bmatrix}
M & 0 \\
0 & 0
\end{bmatrix}Q^\ast : M \in \Herm_r 
\right\rbrace , \quad
H_A \coloneqq  \left\lbrace Q \begin{bmatrix}
0 & B \\
B^\ast & 0
\end{bmatrix}Q^\ast : B \in \mathbb{C}^{r \times (n-r)}
\right\rbrace.
\]
\begin{lemma}\label{lem:connection}
Spaces $V_A $ and $H_A$ are well-defined. Moreover, $V_A = \ker d_A \pi_{n,r}$ and $V_A \oplus H_A = T_A \Herm_{n,r}^+$.
\end{lemma}
\begin{proof}
Let $A = Q' \begin{bmatrix}
M'_A & 0 \\
0 & 0
\end{bmatrix} {Q'}^\ast$ be another decomposition of $A$. Then we have 
\[
X \begin{bmatrix}
M'_A & 0 \\
0 & 0
\end{bmatrix} X^\ast =  \begin{bmatrix}
M_A & 0 \\
0 & 0
\end{bmatrix},
\]
where $X = {Q'}^\ast Q'$. Partitioning $X$ as $X = \begin{bmatrix}
X_1 & X_2 \\
X_3 & X_4
\end{bmatrix}$ where $X_1\in \mathbb{C}^{r \times r}$ and $X_4 \in \mathbb{C}^{(n-r) \times (n-r)}$, we obtain 
\[
\ \begin{bmatrix}
X_1 M'_A X_1^\ast & X_1 M'_A X_3^\ast \\
X_3 M'_A X_1^\ast & X_3 M'_A X_3^\ast   
\end{bmatrix} = \begin{bmatrix}
M_A & 0 \\
0 & 0
\end{bmatrix}.
\]
This implies $X_3 = 0$ and $X_1 M'_A X_1^\ast = M_A$. Moreover, we have $X_2 = 0, X_1\in U(r)$ and $X_4\in U(n-r)$ since $X\in U(n)$. Thus we have 
\begin{align*}
Q' \begin{bmatrix}
M & 0 \\
0 & 0 
\end{bmatrix} {Q}^\ast  &= Q \begin{bmatrix}
X_1 & 0 \\
0 & X_4 
\end{bmatrix} \begin{bmatrix}
M & 0 \\
0 & 0 
\end{bmatrix} \begin{bmatrix}
X^\ast_1 & 0 \\
0 & X^\ast_4 
\end{bmatrix} Q^\ast =  Q \begin{bmatrix}
X_1 M X_1^\ast & 0 \\
0 & 0 
\end{bmatrix} Q^\ast, \\
Q' \begin{bmatrix}
0 & B \\
B^\ast & 0 
\end{bmatrix} {Q'}^\ast &= Q \begin{bmatrix}
X_1 & 0 \\
0 & X_4 
\end{bmatrix} \begin{bmatrix}
0 & B \\
B^\ast & 0 
\end{bmatrix} \begin{bmatrix}
X^\ast_1 & 0 \\
0 & X^\ast_4 
\end{bmatrix} Q^\ast = Q \begin{bmatrix}
0 & X_1 B X_4^\ast \\
X_4 B X_1^\ast & 0
\end{bmatrix} Q^\ast,
\end{align*}
from which we may conclude that $V_A$ and $H_A$ are independent on the decomposition of $A$. 

It is clear that $V_A \subseteq \ker d_A \pi_{n,r}$. By comparing dimensions, we see immediately that $V_A = \ker d_A \pi_{n,r}$. We observe that $H_A \subseteq T_A \Herm^+_{n,r}$ and $V_A \cap H_A = \{0\}$ by definition. Therefore, we must have $V_A \oplus H_A = T_A \Herm^+_{n,r}$ since $\dim_{\mathbb{R}} V_A = r^2, \dim_{\mathbb{R}} H_A = 2 r(n-r)$ and $\dim_{\mathbb{R}} \Herm^+_{n,r} = 2rn - r^2$.
\end{proof}
A direct consequence of Lemma~\ref{lem:connection} is that the assigment $A \mapsto H_A$ defines an Ehresmann connection \cite{Ehresmann95,KN96,Michor08} on $\pi_{n,r}: \Herm^+_{n,r} \to \Gr(r,n)$. With respect to this connection, the \emph{parallel transport} of $A\in \Herm^+_{n,r}$ along a curve $\gamma: [0,1] \to \Gr(r,n)$ is defined to be a curve $P^A_{\gamma}: [0,1] \to \Herm^+_{n,r}$ such that 
\[
P^A_{\gamma}(0) = A,\quad \pi_{n,r} (P^A_{\gamma}(t) ) = \gamma(t),\quad  \frac{d P^A_{\gamma}(t)}{dt} \in H_{P^A_{\gamma}(t)}.
\]

Although on a general fiber bundle, the parallel transport may only exist on a small portion \cite{KN96} of the base curve, it does exist on the whole of $\gamma$ for $\pi_{n,r}: \Herm^+_{n,r} \to \Gr(r,n)$. In fact, we are even able to construct $P_{\gamma}^A$ explicitly. To this end, we let $c:[0,1] \to V(n,r)$ be the horizontal lift of $\gamma$ in $V(r,n)$ at $Q_0 \begin{bmatrix}
I_r &
0
\end{bmatrix}^\tp$, where first $r$ column vectors of $Q_0\in U(n)$ consist an orthonormal basis of $\gamma(0)$. We recall from \cite{EAS98} that there exists some smooth function $B:[0,1] \to \mathbb{C}^{r \times (n-r)}$ such that $B(0) = 0$ and 
\begin{equation}\label{eq:horizontal lift}
c(t) = Q_0 \exp \left(\begin{bmatrix}
0 & - B(t)\\
B(t)^\ast  & 0
\end{bmatrix}  \right) \begin{bmatrix}
I_r \\
0
\end{bmatrix},\quad t\in [0,1].
\end{equation}
\begin{lemma}\label{lem:paralell transport}
If $\gamma:[0,1] \to \Gr(r,n)$ is a smooth curve and $A\in \Herm_{n,r}^+$ satisfies $\pi_{n,r}(A) = \gamma(0)$, then the parallel transport of $A$ along $\gamma$ is 
\begin{equation}\label{lem:paralell transport:eq}
P_{\gamma}^A (t) = c(t) M_{A\mid_{\gamma(0)}} c(t)^\ast, \quad t\in [0,1].
\end{equation}
Here $c(t)$ is the horizontal lift of $\gamma(t)$ as in \eqref{eq:horizontal lift} and $M_{A\mid_{\gamma(0)}} \in \Herm_r^+$ is the matrix representation of $A\mid_{\gamma(0)}$ with respect to the orthonormal basis of $\gamma(0)$ consisting of column vectors of $c(0) = Q_0 \begin{bmatrix}
I_r &
0
\end{bmatrix}^\tp$.
\end{lemma}
\begin{proof}
By definition, we clearly have $P_{\gamma}^A (0) = A$. For each $t\in [0,1]$, we notice that $\ker (P_{\gamma}^A (t) ) = \operatorname{span}\{v_1,\dots, v_{n-r}\}$ where $c(t)^\ast v_j = 0$ for each $j=1,\dots, n-r$, thus $\pi_{n,r} (P_{\gamma}^A (t)) = \gamma(t)$. We also observe that 
\[
( P^A_\gamma )'(t) = Q(t) \begin{bmatrix}
0 & M_A B'(t) \\
{B'(t)}^{\ast} M_A & 0
\end{bmatrix} Q(t)^\ast \in H_{P^A_\gamma(t)},
\] 
where $Q(t) = Q_0 \exp \left(\begin{bmatrix}
0 & - B(t)^\ast \\
B(t) & 0
\end{bmatrix} \right)$. Therefore, $P^A_{\gamma}$ is the parallel transport of $A$ along $\gamma$.
\end{proof}
We remark that $Q_0\in U(n)$ in \eqref{eq:horizontal lift} is not uniquely determined. Thus the curve $c$ varies accordingly if we change $Q_0$. However, a direct calculation implies that $P_{\gamma}^A$ is well-defined for different choices of $Q_0$ since $M_{A\mid_{\gamma(0)}}$ in \eqref{lem:paralell transport:eq} depends on $Q_0$ as well. One can also apply the uniqueness (if it exists) of the parallel transport \cite[Theorem~17.8]{Michor08} for general fiber bundle to see that $P_{\gamma}^A$ does not depend on choices of $Q_0$.

Let $d$ and $\delta$ be geodesic distance functions (cf. Example~\ref{ex:GD}) on $\Gr(r,n)$ and $\Herm^{++}_r$ with respect to their canonical metrics \cite{EAS98,Bhatia15}, respectively. Combining \eqref{eq:GD}, \eqref{prop:simpify delta_H:eq} and Lemma~\ref{lem:paralell transport}, we obtain the following description of $\operatorname{GD}_{d,\delta}$ by parallel transports.
\begin{proposition}\label{prop:paralell transport}
For any $A, B\in \Herm^+_{n,r}$, $\operatorname{GD}_{d,\delta}(A,B)^2$ is the total energy cost of first parallelly transporting $A$  into $\pi_{n,r}^{-1} (\pi_{n,r}(B))$ along a geodesic, then moving the corresponding parallel transport to $B$ inside $\pi_{n,r}^{-1} (\pi_{n,r}(B))$. In other words, we have 
\[
\operatorname{GD}^2_{d,\delta}(A,B) = d^2(\pi_{n,r}(A),\pi_{n,r}(B)) + \delta^2 (\mathcal{A}_1, B\mid_{\pi_{n,r}(B)}),
\]
where we define
\[
\mathcal{A}_1 \coloneqq \left\lbrace 
P^A_{\gamma}(1)\mid_{\pi_{n,r}(B)}: \gamma~\text{is a geodesic curve connecting $\pi_{n,r}(A)$ and $\pi_{n,r}(B)$}
\right\rbrace
\]
and $\delta (\mathcal{A}_1, B_1)$ is the point-set distance between $\mathcal{A}_1$ and $B\mid_{\pi_{n,r}(B)}$ with respect to any orthonormal basis of $\pi_{n,r}(B)$.
\end{proposition}
\subsection{Interpretation of \texorpdfstring{$\operatorname{GD}_{d,\delta}$}{} by quasi-geodesic}\label{subsec:quasi-geodesic}
It is observed in \cite{BS10} that $\Herm^+_{n,r}$ is a quotient space $\Herm^+_{n,r} \simeq \left( V(n,r) \times \Herm^{++}_{r} \right) / U(r)$, from which a family of Riemannian metrics $\{g_k: k > 0 \}$ on $\Herm^{+}_{n,r}$ are induced from those on $V(n,r) \times \Herm^{++}_{r}$. Although geodesic curves with respect to these Riemannian metrics can not be computed explicitly, a quasi-geodesic curve is proposed in \cite{BS10}. For $A,B\in \Herm^+_{n,r}$, such a curve connecting $A$ and $B$ can be written as 
\begin{equation}\label{eq:quasi-geodesic}
\mu(t) \coloneqq u(t) S(t) u(t)^\ast, \quad t\in [0,1],
\end{equation}
where $u(t)$ is the horizontal lift in $V(n,r)$ of a geodesic in $\Gr(r,n)$ connecting $\pi_{n,r}(A)$ and $\pi_{n,r}(B)$, $S(t)$ is a geodesic in $\Herm^{++}_{r}$ connecting the matrix representations of $A$ and $B$ with respect to column vectors of $u(0)$ and $u(1)$. The length of $\mu$ with respect to $g_k$ is 
\[
\operatorname{length}_k(\mu) = \left( d^2(\pi_{n,r}(A),\pi_{n,r}(B)) + k \delta^2 (S(0),S(1)) \right)^{\frac{1}{2}}.
\]
Here $d$ and $\delta$ are the same distance functions as in Proposition~\ref{prop:paralell transport}. According to Corollary~\ref{cor:simpify delta_H}, we obtain the following description of $\operatorname{GD}_{d,\delta}$. 
\begin{corollary}\label{cor:quasi geodesic distance}
Let $A,B, \mu, d$ and $\delta$ be as above. Then we have $\operatorname{GD}_{d,\delta} (A,B) = \operatorname{length}_1 (\mu)$ if principal angles between $\pi_{n,r}(A)$ and $\pi_{n,r}(B)$ are strictly smaller than $\pi/2$. 
\end{corollary}
We remark that if some principal angels between $\pi_{n,r}(A)$ and $\pi_{n,r}(B)$ are equal to $\pi/2$, then $\operatorname{GD}_{d,\delta} (A,B)$ is not necessarily equal to $\operatorname{length}_1 (\mu)$, since geodesics between $\pi_{n,r}(A)$ and $\pi_{n,r}(B)$ are not unique. Instead, we have 
\[
\operatorname{GD}_{d,\delta} (A,B) = \min_{\mu}\{ \operatorname{length}_1 (\mu) \},
\]
where $\mu$ runs through all possible quasi-geodesics defined in \eqref{eq:quasi-geodesic}.
\section{Distance between non-equidimensional PD matrices}\label{sec:PD}
According to Definition~\ref{def:GD}, the geometric distance is defined for $A \in \Herm_{n,r}^+$ and $B\in \Herm_{m,s}^+$ as 
\[
\operatorname{GD}_{d,\delta}(A,B) \coloneqq \left( d^2(\pi_{n,r}(A),\iota_{m,n} \circ \pi_{m,s}(B)) + \delta_H^2 (\mathcal{Z}_{A_{\mid \pi_{n,r}(A)}, B_{\mid \iota_{m,n} \circ \pi_{m,s}(B)}}) \right)^{\frac{1}{2}},
\]
where $d: \Gr(r,n) \times \Gr(s,n) \to \mathbb{R}_+$, $\delta: \Herm_r^{++} \times \Herm_s^{++} \to \mathbb{R}_+$ are given functions, $\delta_H$ is the generalized Hausdorff function defined in \eqref{eq:Hausdorff} and $\mathcal{Z}_{A_{\mid \pi_{n,r}(A)}, B_{\mid \iota_{m,n} \circ \pi_{m,s}(B)}}$ is the set defined for $A_{\mid \pi_{n,r}(A)}$ and $B_{\iota_{m,n} \circ \pi_{m,s}(B)}$ in \eqref{eq:ZXY}. Therefore properties of $d$ and $\delta$ are crucial to those of $\operatorname{GD}_{d,\delta}$. For example, $\operatorname{GD}_{d,\delta}$ can be simplified if $\delta$ has some invariant properties (cf.  Proposition~\ref{prop:simpify delta_H} and Corollary~\ref{cor:simpify delta_H}). In \cite[Theorem~12]{ye2016schubert}, a list of choices of $d$ is given, which can be easily computed by principal angles and generalize the commonly used distance functions on Grassmann manifolds. As for choices of $\delta$, an extension of the geodesic distance on $\Herm^+_r$ with respect to its canonical Riemannian metric is proposed in \cite{lim2019geometric}. Methods based on the optimal transport between Gaussian distributions are discussed in \cite{AJA22,CL22}. Except for these, we are not aware of other existing works concerning possible choices of $\delta$. 

The goal of this section is to extend commonly used distance \cite{TP21,Sra12}, or more generally, divergence \cite{Stein64,CCA15,CZA06,KSD09} functions on $\Herm_r^{++}$ to non-equidimensional cases. For computational purposes, these extensions should be easily computable, just as their equidimensional counterparts do. We follow the method in \cite{lim2019geometric} to relate elements in $\Herm^{++}_r$ and $\Herm^{++}_s$ by ellipsoids associated to them. To be more precise, given $Z\in \Herm^{++}_m$, we define its associated ellipsoid by  
\[
E_Z \coloneqq \{x\in \mathcal{k}^m: x^\ast Z x \le 1 \}.
\]
Without loss of generality, we assume $r \le s$. For $C \in \Herm^{++}_r$ and $D \in \Herm^{++}_s$, we consider
\begin{align*}
\Omega_+(C) &\coloneqq \left\lbrace
Y\in \Herm_s^{++}: E_C \subseteq E_Y
\right\rbrace = \left\lbrace
Y\in \Herm_s^{++}: Y_{11} \preceq C
\right\rbrace, \\
\Omega_-(D) &\coloneqq \left\lbrace
X\in \Herm_r^{++}: E_X \subseteq E_D
\right\rbrace = \left\lbrace
X \in \Herm_r^{++}: D_{11} \preceq X
\right\rbrace.
\end{align*}

Let $\Delta_m: \Herm_m^{++} \times \Herm_m^{++} \to \mathbb{R}$ be a function of the form 
\begin{equation}\label{eq:delta}
\Delta_m (Z,W) = g ( \lambda_1 (f(Z,W)),\dots,  \lambda_m (f(Z,W))),\quad (Z,W) \in \Herm_m^{++} \times \Herm_m^{++}.
\end{equation}
Here $f: \Herm_m^{++} \times \Herm_m^{++} \to \Herm_m$ is a map, $g: \mathbb{R}^m \to \mathbb{R}$ is a function and $ \lambda_1(Z) \ge \dots \ge \lambda_m(Z)$ are eigenvalues of $Z\in \Herm_m$. 
\begin{lemma}\label{lem:Delta-}
Let $\Delta_r: \Herm_r^{++} \times \Herm_r^{++} \to \mathbb{R}$ be a function as in \eqref{eq:delta}. For $C \in \Herm_r^{++}$ and $D \in \Herm_s^{++}$, if $f(C,\cdot) : \Herm_r^{++} \to \Herm_r^{++}$ is surjective and order preserving in the sense that 
\[
X \preceq X' \implies f(C, X) \le f(C,X'),
\]
then 
\[
\Delta_r(C,\Omega_{-}(D))  \coloneqq \min_{X\in \Omega_{-}(D)} \Delta_r(C,X) = 
\min_{\substack{ t_1 \ge \cdots  \ge t_r  \\  t_k \ge \lambda_k(f(C,D_{11})), 1 \le k \le r }} g(t_1,\dots, t_j).
\]
\end{lemma}
\begin{proof}
By definition, $X \in \Omega_{-}(D)$ if and only if $X \succeq D_{11}$. Since $f(C,\cdot)$ is order preserving, we have $ f(C,X) \succeq f(C,D_{11})$. This implies that $\lambda_k(f(C,X)) \ge \lambda_k(f(C,D_{11}))$, from which we obtain the desired formula for $\delta(C,\Omega_{-}(D))$ as $f$ is surjective.
\end{proof}
\begin{lemma}\label{lem:Delta+}
Let $\Delta_s: \Herm_s^{++} \times \Herm_s^{++} \to \mathbb{R}$ be a function as in \eqref{eq:delta}. We assume that there exists some surjective increasing function $h: \mathbb{R} \to \mathbb{R}$ such that $\lambda_j (f(I_s,Y)) = h(\lambda_j (Y))$ for any $1\le j \le s$ and $Y \in \Herm_s^{++}$. We further assume that $\Delta_s$ is congruence invariant, i.e., $\Delta_s(G X G^\tp, G Y G^\tp) = \Delta(X,Y)$, $X,Y\in \Herm_s^{++}$ and $G \in \GL_s(\mathcal{k})$. For $C \in \Herm_r^{++}$ and $D\in \Herm_s^{++}$, we have 
\[
\Delta_s (\Omega_{+}(C),D) \coloneqq \min_{Y\in \Omega_+(C)} \Delta_s (Y, D)  =   \min_{\substack{ 
t_1 \ge \cdots \ge t_s, \\
h(\lambda_k(D_{11}^{-1} C)) \ge t_{j_k}, ~\text{for some}~1 \le j_1 < \cdots < j_r \le s } } 
g(t_1,\dots, t_s).
\]
\end{lemma}
\begin{proof}
Since $\Delta_s$ is congruence invariant, by \cite[Section~V--B]{lim2019geometric} we may choose some $G \in \GL_n(\mathcal{k})$
\[
\delta(\Omega_+(C),D) = \delta (G \Omega_+(C) G^\ast,G D G^\ast) = \delta(\Omega_+(\Sigma), I_s),
\]
where $\Sigma = \diag(\lambda_1(D_{11}^{-1} C),\dots, \lambda_r(D_{11}^{-1}C))$. Thus we have 
\begin{align*}
\delta(\Omega_+(C),D)  &= \delta(I_s, \Omega_{+}(\Sigma)) \\
&= \min_{Y\in \Omega_{+}(\Sigma)} \delta(I_s,Y) \\
&= \min_{\Sigma \succeq Y_{11}} g(\lambda_1(f(I_s,Y)),\dots, \lambda_n(f(I_s,Y))) \\
&= \min_{ \Sigma \succeq Y_{11} } 
g(h(\lambda_1(Y)),\dots, h(\lambda_s(Y)))\\
&= \min_{\substack{t_j = h( \lambda_j(Y) ), \\ \Sigma \succeq Y_{11}} } g(t_1,\dots, t_s) \\
& = \min_{\substack{ 
t_1 \ge \cdots \ge t_s, \\
h(\lambda_k(D_{11}^{-1} C)) \ge t_{j_k}, ~\text{for some}~1 \le j_1 < \cdots < j_r \le s
}} g(t_1,\dots, t_s)
\end{align*}
The last equality holds since the sets 
\begin{align*}
S_1 &\coloneqq \lbrace 
(t_1,\dots, t_s)\in \mathbb{R}^s: t_j = h(\lambda_j(Y)),~1 \le j \le s,~\Sigma \succeq Y_{11}
\rbrace,  \\
S_2 &\coloneqq \lbrace
(t_1,\dots, t_s)\in \mathbb{R}^s: t_1 \ge \cdots \ge t_s,~h(\lambda_k(D_{11}^{-1} C)) \ge t_{j_k},~\text{for some}~1 \le j_1 < \cdots < j_r \le s
\rbrace
\end{align*}
are exactly the same. Indeed, $S_1 \subseteq S_2 $ since $h$ is increasing and $\lambda_k(\Sigma) \ge \lambda_k(Y_{11}) \ge \lambda_k(Y)$ for each $1 \le k \le r$. Conversely, for any $(t_1,\dots, t_s) \in S_2$, we take $Y = \diag (\tau_{j_1},\dots, \tau_{j_r}, \tau_{j_r + 1} ,\dots, \tau_{j_s})$ where $\tau_j = h^{-1}(t_j), 1 \le j \le s$ and $\{ j_{r+1},\dots, j_s \} = \{r+1,\dots, s\}$. We notice that the existence of $Y$ is ensured by the surjectivity of $h$. Moreover, we must have $\tau_1 \ge \cdots \ge \tau_s$ by the monotonicity of $h$. This indicates that $\lambda_j(Y) = \tau_j, 1 \le j \le s$. We also observe that $\tau_{j_k} \le \lambda_k (D^{-1}_{11}C)$. Thus  $Y_{11} = \diag (\tau_{j_1},\dots, \tau_{j_r}) \preceq \Sigma$ and this implies $S_2 \subseteq S_1$.
\end{proof}
\subsection{Consistency of geodesic distance}\label{subsec:geodist}
As a direct consequence of \cite[Theorem~ 2.1]{TP21}, we have the following characterization of congruence invariant geodesic distances on $\Herm_n^{++}$.
\begin{lemma}\label{lem:characterization}
Let $m$ be a positive integer and let $\delta: \Herm_m^{++} \times  \Herm_m^{++} \to \mathbb{R}$ be a geodesic distance function with respect to some Riemannian metric on $\Herm_m^{++}$. Then $\delta$ is congruence invariant if and only if there exist some constants $\alpha > 0$ and $\beta > -\alpha/m$ such that 
\[
\delta(X,Y) =\sqrt{ \alpha \lVert  \log(X^{-1}Y) \rVert^2 + \beta \log^2 \det(X^{-1}Y) }.
\]
\end{lemma} 
\begin{theorem}[consistency of geodesic distance]\label{thm:invariant distance}
Let $r \le s$ be positive integers and let $\alpha > 0$, $\beta > -\alpha/s$ be some fixed real numbers. Suppose 
\begin{align*}
\delta_r (X,Y) &= \sqrt{ \alpha \lVert  \log(X^{-1}Y) \rVert^2 + \beta \log^2 \det(X^{-1}Y) },\quad (X,Y) \in \Herm_r^{++} \times \Herm_r^{++}, \\
\delta_s (Z,W) &=\sqrt{ \alpha \lVert  \log(Z^{-1}W) \rVert^2 + \beta \log^2 \det(Z^{-1}W) },\quad (Z,W) \in \Herm_s^{++} \times \Herm_s^{++}.
\end{align*}
Then for each pair $(C,D) \in \Herm_r^{++} \times \Herm_s^{++}$ we have 
\[
\delta_s(\Omega_{+}(C),D)  \le \delta_r (C, \Omega_{-}(D)).
\]
Moreover, $\delta_s(\Omega_{+}(C),D)  = \delta_r (C, \Omega_{-}(D))$ for all $(C,D) \in \Herm_r^{++} \times \Herm_s^{++}$ if and only if either $\beta = 0$ or $r = s$.
\end{theorem}
\begin{proof}
We observe that 
\small
\begin{align*}
\delta_r^2(X,Y) &= \alpha \sum_{k=1}^r \log^2(\lambda_k(X^{-1}Y)) + \beta \left(  \sum_{k=1}^r \log(\lambda_k(X^{-1}Y)) \right)^2 \\
&= \begin{bmatrix}
\log (\lambda_1(X^{-1}Y)) & \cdots & \log (\lambda_r(X^{-1}Y))
\end{bmatrix} \begin{bmatrix}
\alpha + \beta & \beta & \cdots & \beta \\
\beta & \alpha + \beta & \cdots & \beta \\
\vdots & \vdots & \ddots & \vdots \\
\beta & \beta & \cdots & \alpha +\beta 
\end{bmatrix} \begin{bmatrix}
\log (\lambda_1(X^{-1}Y)) \\
\vdots \\
\log (\lambda_r(X^{-1}Y))
\end{bmatrix}.
\end{align*}\normalsize
Thus according to Lemma~\ref{lem:Delta-}, we have 
\begin{equation}\label{thm:invariant distance:eq:delta-}
\delta^2_r(C,\Omega_{-}(D)) = \min_{\substack{ t_1 \ge \cdots \ge t_r  \\ t_k \ge \log (\lambda_k)}} \left( \begin{bmatrix}
t_1 & \cdots & t_r 
\end{bmatrix} \begin{bmatrix}
\alpha + \beta & \beta & \cdots & \beta \\
\beta & \alpha + \beta & \cdots & \beta \\
\vdots & \vdots & \ddots & \vdots \\
\beta & \beta & \cdots & \alpha +\beta 
\end{bmatrix} \begin{bmatrix}
t_1 \\
\vdots \\
t_r
\end{bmatrix} \right).
\end{equation}
Here for brevity we denote $\lambda_k \coloneqq \lambda_k(C^{-1}D_{11}),1 \le k \le r$. Similarly, by Lemma~\ref{lem:Delta+} and the fact that $\lambda_k(Z^{-1}) = \lambda_{m-k+1} (Z)^{-1}$ for any $Z\in \Herm_m^{++}$ and $1 \le k \le m$, we may also write $\delta^2_s(\Omega_{+}(C),D)$ as 
\begin{align*}
& \min_{ \substack{t_1 \ge \cdots \ge t_s , \\  
-\log(\lambda_{r-k+1}) \ge t_{j_k}, ~\text{some}~1 \le j_1 < \cdots < j_r \le s 
}} \left( \begin{bmatrix}
t_1 & \cdots & t_s 
\end{bmatrix} \begin{bmatrix}
\alpha + \beta & \beta & \cdots & \beta \\
\beta & \alpha + \beta & \cdots & \beta \\
\vdots & \vdots & \ddots & \vdots \\
\beta & \beta & \cdots & \alpha +\beta 
\end{bmatrix} \begin{bmatrix}
t_1 \\
\vdots \\
t_s
\end{bmatrix} \right) \\
= & \min_{\substack{\tau_{1} \ge \dots \ge \tau_s , \\
\tau_{s - j_k + 1} \ge \log(\lambda_{r - k + 1}),~\text{some}~1 \le j_1 < \cdots < j_r \le s
}} \left( \begin{bmatrix}
\tau_s & \cdots & \tau_{1}
\end{bmatrix} \begin{bmatrix}
\alpha + \beta & \beta & \cdots & \beta \\
\beta & \alpha + \beta & \cdots & \beta \\
\vdots & \vdots & \ddots & \vdots \\
\beta & \beta & \cdots & \alpha +\beta 
\end{bmatrix} \begin{bmatrix}
\tau_s \\
\vdots \\
\tau_{1}
\end{bmatrix} \right).
\end{align*}
Here the last equality is obtained by taking $\tau_{s-j+1} = -t_j, 1\le j \le s$. Since the objective function in the above expression of $\delta^2_s(\Omega_{+}(C),D)$ is invariant under permutations of variables, we may further rewrite 
\begin{equation}\label{thm:invariant distance:eq:delta+}
\delta^2_s(\Omega_{+}(C),D) = 
\min_{\substack{\tau_{1} \ge \dots \ge \tau_r \\
\tau_{k} \ge \log(\lambda_k)}} \left( \begin{bmatrix}
\tau_{1} & \cdots & \tau_{s} 
\end{bmatrix} \begin{bmatrix}
\alpha + \beta & \beta & \cdots & \beta \\
\beta & \alpha + \beta & \cdots & \beta \\
\vdots & \vdots & \ddots & \vdots \\
\beta & \beta & \cdots & \alpha +\beta 
\end{bmatrix} \begin{bmatrix}
\tau_1 \\
\vdots \\
\tau_s
\end{bmatrix} \right).
\end{equation}

On the other hand, for any $t'\in \mathcal{k}^r, t^{''}\in \mathcal{k}^{s-r}, A' \in \mathcal{k}^{r\times r}, b\in \mathcal{k}^{r \times (s-r)}$ and $A^{''} \in \mathcal{k}^{(s-r) \times (s-r)}$, it holds that 
\[
\begin{bmatrix}
t' & t^{''}
\end{bmatrix} \begin{bmatrix}
A' & b \\
b^\ast & A^{''} 
\end{bmatrix} \begin{bmatrix}
{t'}^\ast \\ {t^{''}}^\ast
\end{bmatrix} = t' A' {t^{'}}^\ast + 2(t^{''}b^\ast {t'}^{\ast}) + t^{''} A^{''} {t^{''}}^\ast.
\]
This implies $\begin{bmatrix}
t' & 0
\end{bmatrix} \begin{bmatrix}
A' & b \\
b^\tp & A^{''} 
\end{bmatrix} \begin{bmatrix}
{t'}^\ast \\0
\end{bmatrix} = t' A' {t^{'}}^\ast$ and thus $\delta_s(\Omega_{+}(C),D) \le \delta_r(C,\Omega_{-}(D))$ according to \eqref{thm:invariant distance:eq:delta-} and \eqref{thm:invariant distance:eq:delta+}. Clearly the equality holds for any $(C,D) \in \Herm_r^{++} \times \Herm_s^{++}$ if $s = r$. The equality for $\beta = 0$ is proved in \cite{lim2019geometric}.

Assume that $t' = (t_1,\dots, t_r)$ is chosen so that $t_k \ge \log(\lambda_k), 1\le k \le r$ and  
\[
\begin{bmatrix}
t_1 & \cdots & t_r 
\end{bmatrix} \begin{bmatrix}
\alpha + \beta & \beta & \cdots & \beta \\
\beta & \alpha + \beta & \cdots & \beta \\
\vdots & \vdots & \ddots & \vdots \\
\beta & \beta & \cdots & \alpha +\beta 
\end{bmatrix} \begin{bmatrix}
t_1 \\
\vdots \\
t_r
\end{bmatrix}  = \delta_{r}(C, \Omega_{-}(D)).
\]
We consider $t^{''} = (t_{r+1},\dots, t_s) = (x, \dots, x)$, $b = \begin{bmatrix}
\beta & \cdots & \beta \\
\vdots & \ddots & \vdots \\
\beta & \cdots & \beta 
\end{bmatrix}\in \mathcal{k}^{r \times (s-r)}$ and 
\[
A' = \begin{bmatrix}
\alpha + \beta & \beta & \cdots & \beta \\
\beta & \beta & \cdots & \beta \\
\vdots & \vdots & \ddots & \vdots \\
\beta & \beta & \cdots & \alpha + \beta
\end{bmatrix} \in \mathcal{k}^{r \times r}, \quad A^{''} = \begin{bmatrix}
\alpha + \beta & \beta & \cdots & \beta \\
\beta & \beta & \cdots & \beta \\
\vdots & \vdots & \ddots & \vdots \\
\beta & \beta & \cdots & \alpha + \beta
\end{bmatrix} \in \mathcal{k}^{(s-r)\times (s - r)}.
\] 
Then 
\begin{align*}
\delta_s(\Omega_{+}(C),D)
&= \begin{bmatrix}
t' & t^{''}
\end{bmatrix} \begin{bmatrix}
A' & b \\
b^\tp & A^{''} 
\end{bmatrix} \begin{bmatrix}
{t'}^\tp \\ {t^{''}}^\tp
\end{bmatrix} \\ 
&= t' A' {t^{'}}^\tp + 2(t^{''}b^\tp {t'}^{\tp}) + t^{''} A^{''} {t^{''}}^\tp \\
&\le \delta_r(C,\Omega_{-}(D))  + \left( 2(s - r)\beta \sum_{k=1}^r t_k \right) x + (s-r)((s-r)\beta + \alpha) x^2.
\end{align*}
If $r < s$ and $\beta \ne 0$, then we take $C,D$ such that $\log (\lambda_k) > 0$ for all $1\le k \le r$. This ensures the existence of some $x\in \mathbb{R}$ such that 
\[
\left( 2(s-r)\beta \sum_{k=1}^r t_k \right) x + (s-r)((s-r)\beta + \alpha) x^2  < 0.
\]
Therefore we obtain $\delta_s(\Omega_{+}(C),D) < \delta_r(C,\Omega_{-}(D))$.
\end{proof}
\subsection{Consistency of divergences}\label{subsec:diverg}
It is proved in \cite{lim2019geometric} that $\delta_s(\Omega_{+}(C),D) = \delta_r(C,\Omega_{-}(D)$ holds if we take $\delta_r$ (resp. $\delta_s$) to be the geodesic distance $\lVert \log (X^{-1}Y) \rVert$ for $X,Y\in \Herm_r^+$ (resp. $X,Y\in \Herm_s^+$). Theorem~\ref{thm:invariant distance} implies that such a geodesic distance is unique (up to a positive scaling) if it is congruence invariant. However, a measurement of the nearness of two positive definite matrices is not necessarily a distance function. In the literature, such a measurement is called a divergence \cite{CCA15,
KSD09,
Stein64,
Sra12,
CZA06}. In the next theorem, we prove that most commonly-used divergences satisfy the relation $\delta(C, \Omega_{-}(D)) = \delta(D, \Omega_+(C))$ for all $(C,D) \in \Herm_r^{+} \times \Herm_s^{+}$.
\begin{theorem}[consistency of divergences]\label{thm:divergence}
Let $r \le s$ be positive integers  and let 
\[
\delta_r: \Herm_r^{++} \times \Herm_r^{++} \to \mathbb{R}, \quad \delta_s: \Herm_s^{++} \times \Herm_s^{++} \to \mathbb{R}
\]
be divergences of one of the following types: 
\begin{enumerate}[(a)]
\item $\alpha\beta$ log-det divergence \cite{CCA15}: $\frac{1}{\alpha \beta} \log \det \left( 
\frac{\alpha (X^{-1}Y)^\beta + \beta (Y^{-1}X)^\alpha}{\alpha  + \beta}
\right)$, where $\alpha, \beta$ are fixed nonzero real numbers such that $ \alpha + \beta \ne 0$. \label{thm:divergence:item1}
\item Generalized Stein's loss \cite{KSD09,Stein64}: $\frac{1}{\alpha^2}\left[ \tr( (X Y^{-1})^\alpha) - \log(\det ( (X Y^{-1})^{\alpha} )) - n \right]$, where $\alpha \ne 0$ is a fixed real number and $n$ is the dimension of matrices $X$ and $Y$.\label{thm:divergence:item2}
\item Generalized Burg divergence \cite{kulis2006learning}: $\frac{1}{\alpha^2}\left[ \tr( (X^{-1} Y)^\alpha) - \log(\det ( (X^{-1} Y)^{\alpha} )) - n \right]$, where $\alpha \ne 0$ is a fixed real number \label{thm:divergence:item3} and $n$ is the dimension of matrices $X, Y$.
\item Generalized Itakura-Saito log-det divergence \cite{CCA15,CZA06}:
$\frac{1}{\alpha^2}\log \frac{\det(XY^{-1})^\alpha}{\det(I_n + \log (XY^{-1})^\alpha)}$, where $\alpha \ne 0$ is a fixed real number.\label{thm:divergence:item4}
\item Geodesic distance: $\lVert \log (X^{-1}Y) \rVert$. \label{thm:divergence:item5}
\end{enumerate}
We also denote by $f_{\operatorname{sym}}$ the symmetrization of a function $f: \Herm_m^{++} \times \Herm_m^{++} \to \mathbb{R}$, i.e.,
\[
f_{\operatorname{sym}}(X,Y) \coloneqq \frac{1}{2} \left( f(X,Y) +  f(Y,X) \right). 
\]
Then for any $(C,D) \in \Herm_r^{++} \times \Herm_{s}^{++}$, we have 
\[
\delta_r(C,\Omega_{-}(D) ) = \delta_s(\Omega_{+}(C),D), \quad (\delta_r)_{\operatorname{sym}}(C,\Omega_{-}(D) ) = (\delta_s)_{\operatorname{sym}}(\Omega_{+}(C),D).
\]
Moreover, both $\delta_r(C,\Omega_{-}(D) )$ and $(\delta_r)_{\operatorname{sym}}(C,\Omega_{-}(D) )$ can be written as explicit functions of 
\[
\max\{1, \lambda_1(C^{-1}D_{11})\}, \dots, \max\{1,\lambda_r(C^{-1}D_{11})\}.
\] 
\end{theorem}
Before we proceed to the proof of Theorem~\ref{thm:divergence}, it is worthy to remark that divergences in \eqref{thm:divergence:item2}--\eqref{thm:divergence:item5} are limits (in an appropriate sense) of \eqref{thm:divergence:item1} as $\alpha,\beta$ or $\alpha + \beta$ approaches to zero. Therefore intuitively one may regard our conclusions for \eqref{thm:divergence:item2}--\eqref{thm:divergence:item5} as consequences of that for \eqref{thm:divergence:item1}. However, taking limits does not supply a rigorous proof since the convergence of \eqref{thm:divergence:item1} to \eqref{thm:divergence:item2}--\eqref{thm:divergence:item5} is not necessarily uniform. To avoid the discussion of the uniform convergence, we provide below a proof which deals with all cases simultaneously.
\begin{proof}
For notational simplicity, we denote both $\delta_r$ and $\delta_s$ by $\delta$ when sizes of matrices are understood from the context. We observe that in each of these cases, $\delta$ can be written as 
\[
\delta (X,Y) = \left( \sum_{j=1}^n g(\lambda_j(X^{-1}Y)) \right)^{a}
\]
for some function $g:\mathbb{R}_+ \to \mathbb{R}$ and $a\in \{1,\frac{1}{2}\}$. The minimizer $t_{\min} = 1$ of $g$ is unique. Moreover $g(1) = 0$ and 
\[
g'(t) = \begin{cases}
> 0,~\text{if}~t >1, \\
< 0,~\text{if}~t <1. 
\end{cases}
\]
In this case, Lemmas~\ref{lem:Delta-} and \ref{lem:Delta+} and the permutation invariance of $\sum_{j=1}^n g(t_j)$ imply the formulae that follow.
\begin{align*}
\delta(C,\Omega_{-}(D)) = \inf_{\substack{t_1 \ge \cdots \ge t_r \\ t_k \ge \lambda_k, 1\le k \le r}} \left( \sum_{k=1}^r g(t_k) \right)^{a}, \\
\delta(\Omega_+(C),D) = \inf_{\substack{ t_1 \ge \cdots \ge t_r \\  t_k \ge \lambda_k, 1\le k \le r}} \left( \sum_{j=1}^s g(t_j) \right)^{a}.
\end{align*}
Here we denote $\lambda_k \coloneqq \lambda_k(A^{-1}B_{11}),1 \le k \le r$. It is straightforward to verify that $\delta(C,\Omega_{-}(D)) \ge \delta(\Omega_{+}(C),D)$. We claim that 
\[
\delta(C, \Omega_{-}(D)) = \left( \sum_{k=1}^r g(\max\{1, 
\lambda_k\}) \right)^a = \delta(\Omega_{+}(C),D).
\]
In fact, we may take 
\[
(t_1, \dots, t_s) = (\max\{1, 
\lambda_1\}, \dots, \max\{1, 
\lambda_r\}, \underbrace{1,\dots, 1}_{\text{$(s-r)$ times}}).
\]
Clearly relations $t_1 \le \cdots \le t_r$ and $\lambda_k \le  t_k ,1\le k \le r$ hold. Thus 
\[
\left( \sum_{k=1}^r g(t_k) \right)^a \ge \delta(C, \Omega_{-}(D)) \ge  \delta(\Omega_{+}(C), D) \ge \min_{\substack{ t_k \ge \lambda_k \\ 1\le k \le r}} \left( \sum_{j=1}^r g(t_j) \right)^a = \left(  \sum_{j=1}^r \min_{\substack{t_k \ge \lambda_k \\ 1\le k \le r }} g(t_j) \right)^a.
\]
\end{proof}
Other than limiting cases stated in Theorem~\ref{thm:divergence} \eqref{thm:divergence:item2}--\eqref{thm:divergence:item5}, the $\alpha\beta$ log-det divergence and its symmetrization also cover other well-known divergences including: 
\begin{itemize}
\item Bhattacharyya divergence \cite{Bhattacharyya43} or S-divergence \cite{Sra12}: $\alpha = \beta =\pm 1/2$.
\item R\'{e}nyi divergence \cite{JC82} or $\alpha$ log-det divergence \cite{chebbi2012means}: $0 < \alpha < 1, \beta = 1- \alpha$.
\item $\beta$ log-det divergence \cite{cichocki2010families}: $\alpha =1, \beta > 0$.
\item Kullback-Leibler divergence \cite{SR51}: $\alpha = 1$, $\beta = 0$.
\end{itemize}
We notice that if $\delta$ is a divergence of one of the types in Theorem~\ref{thm:divergence}, then $\delta$ is unbounded. However, one can construct a bounded divergence from $\delta$ by composing it with a bounded increasing function $h: \mathbb{R}_+ \to \mathbb{R}_+$. Standard choices of $h$ include:
\begin{itemize}
\item $h(t) = \frac{t}{1 + t}$.
\item $h(t) = \min \{\varepsilon, t\}$ where $\varepsilon$ is a fixed positive number.
\end{itemize}
The corollary that follows is a direct consequence of Theorem~\ref{thm:divergence}.
\begin{corollary}
Let $h: \mathbb{R}_+ \to \mathbb{R}_+$ be an increasing function and let $\delta$ be a divergence of one of the types in Theorem~\ref{thm:divergence}. For any $(C,D) \in \Herm_r^{++} \times \Herm_s^{++}$, we have 
\[
h\circ \delta_r (C, \Omega_{-}(D)) = h \circ \delta_s(\Omega_+(C),D)
\]
\end{corollary}
\subsection{Projection and lift}\label{subsec:proj and lift}
Since $\Omega_-(D)$ (resp. $\Omega_+(C)$) is closed in $\Herm^{++}_m$ (resp. $\Herm^{++}_n$), the point-set distance $\delta_r(C, \Omega_-(D))$ (resp. $\delta_s(D, \Omega_+(C))$)  can be achieved by some point in $\Omega_-(D)$ (resp. $\Omega_+(C)$), if the function $\delta_r(C,\cdot)$ (resp. $\delta_s(D,\cdot)$) is continuous. 
\begin{proposition}[uniqueness of projection]\label{prop:projection}
Let $r \le s$ be positive integers and let 
\[
\delta_r: \Herm_r^{++} \times \Herm_r^{++} \to \mathbb{R}
\] 
be a divergence of one of the types in Theorem~\ref{thm:divergence}. For any $(C,D)\in \Herm_r^{++} \times \Herm_s^{++}$, there is a unique $D_{-} \in \Omega_{-}(D)$ such that 
\[
\lambda_k (C^{-1/2} D_{-} C^{-1/2}) = \max \{1, \lambda_k(C^{-1} D_{11}) \},\quad 1 \le k \le r.
\]
As a consequence, we have $\delta_r(C, D_{-}) = \delta_r(C, \Omega_{-}(D))$. 
\end{proposition}
\begin{proof}
Since $\lambda_k (C^{-1/2} D_{-} C^{-1/2}) = \max \{1, \lambda_k(C^{-1} D_{11}) \}, 1 \le k \le r$, we may write 
\[
D_{-}  =  C^{1/2} Q^\ast \Lambda  Q C^{1/2}, 
\]
where $\Lambda = \diag ( \max \{1, \lambda_1(C^{-1} D_{11}) \}, \dots, \max \{1, \lambda_r(C^{-1} D_{11}) \})$ and $Q\in U(r)$. We observe that $C^{-1/2} D_{11} C^{-1/2} \preceq C^{-1/2} D_{-} C^{-1/2} = Q^\ast \Lambda  Q$ if and only if $D_{11} \preceq D_{-}$. Let $P^\ast \Sigma P = C^{-1/2} D_{11} C^{-1/2}$ be an SVD of $C^{-1/2} D_{11} C^{-1/2}$. We have
\[
\Sigma \preceq (QP^\ast)^\ast \Lambda  (QP^\ast).
\]
Thus the uniqueness of $D_{-}$ follows from the uniqueness of the matrix $X = R \Lambda R^\ast$ such that $\Sigma \preceq X$, where $R\in U(r)$. Indeed, if $X$ is unique and $D'_{-}$ is another element in $\Omega_{-}(D)$ such that $\delta_r(C,D'_{-}) = \delta_r(C,\Omega_{-}(D))$, then we may write $D'_{-}  =  C^{1/2} {Q'}^\ast \Lambda  {Q'} C^{1/2}$ for some $Q' \in U(r)$. But this implies that 
\[
(QP^\ast)^\ast \Lambda  (QP^\ast) = X  = (Q'P^\ast)^\ast \Lambda  (Q'P^\ast),
\] 
thus $D_{-}  =  C^{1/2} {Q}^\ast \Lambda  {Q} C^{1/2}  =  C^{1/2} {Q'}^\ast \Lambda  {Q'} C^{1/2} = D'_{-}$.

It is left to prove the uniqueness of $X$. For convenience, we reorder elements in $\Sigma$ and $X$ so that
\[
\Lambda =\diag ( \underbrace{\mu_1,\dots, \mu_1}_{\text{$m_1$ times}},\dots, \underbrace{\mu_l,\dots, \mu_l}_{\text{$m_l$ times}} ),\quad 1\le \mu_1 < \cdots < \mu_l
\]
and $\Sigma \preceq X$ still holds. If $l = 1$ then $\Lambda = \mu_1 I_r$ and $X$ is clearly unique. If $l \ge 2$, then $\lambda_k(X) = \lambda_k(\Lambda) = \lambda_k(\Sigma)$ whenever $k \ge m_1 + 1$. We claim that $R = \diag(R_1,\dots, R_l)$ where $R_p \in U(m_p), 1\le p \le l$. Otherwise, there exists $2 \le l_0 \le l$ such that $R = \diag(R', R_{l_0+1},\dots, R_{l})$ but 
\[
R' \in U (M_{l_0}) \setminus (U(M_{l_0 - 1}) \times U(m_{l_0}) ),
\]
where $R_p \in U(m_p)$ and $M_{p} \coloneqq \sum_{s=1}^{p} m_s$ for $l_0 \le p \le l$. In this case, the $\left( M_{l_0}, M_{l_0} \right)$-th element of $X = R \Lambda R^\ast$ is stricly smaller than $\mu_{l_0}$. Since $\mu_{l_0} > 1$, $\mu_{l_0}$ is an eigenvalue of $C^{-1} D_{11}$ and it is the $\left( M_{l_0}, M_{l_0} \right)$-th element of $\Sigma$. But this contradicts to the assumption that $\Sigma \preceq X$.
\end{proof}
According to Proposition~\ref{prop:projection}, there is a well-defined map 
\[
\pi_{-}: \Herm_r^{++} \times \Herm_s^{++} \to \Herm_r^{++},\quad \pi_{-}(C,D) = D_{-},
\] 
such that the diagram in Figure~\ref{fig:unique projection} commutes. Here $\operatorname{proj}_1$ is the projection map onto the first factor and we denote a divergence on $\Herm_r^{++} \times \Herm_r^{++}$ and its induced function on $\Herm_r^{++} \times \Herm_s^{++}$
by $\delta_r$ and $\delta_r(\cdot, \Omega_{-}(\cdot))$, respectively. 
\begin{figure}[!hbt]
\centering
  \begin{tikzcd}[row sep=1.5cm, column sep=2cm]
      \Herm_r^{++} \times \Herm_s^{++}  \arrow{d}[left]{ (\operatorname{proj}_1,  \pi_{-}) } \arrow{r}{\delta_r(\cdot, \Omega_-(\cdot))}  & \mathbb{R}   \\
    \Herm_r^{++}  \times \Herm_r^{++}  \arrow{ru}[below]{\delta_r} &    
    \end{tikzcd}
  \caption{unique projection}
    \label{fig:unique projection}
\end{figure}
\begin{corollary}\label{cor:projection}
Let $r \le s$ be positive integers and let 
\[
\delta_r: \Herm_r^{++} \times \Herm_r^{++} \to \mathbb{R}
\] 
be a divergence of one of the types in Theorem~\ref{thm:divergence}. For any $(C,D)\in \Herm_r^{++} \times \Herm_s^{++}$, we have 
\begin{equation}\label{cor:projection:eq:B-}
\pi_-(C,D) = D_{-}  =  C^{1/2} Q^\ast \Lambda  Q C^{1/2}, 
\end{equation}
where $\Lambda = \diag ( \max \{1, \lambda_1(C^{-1} D_{11}) \}, \dots, \max \{1, \lambda_r(C^{-1} D_{11}) \})$ and $Q\in U(r)$ is obtained by an SVD of $C^{-1/2} D_{11} C^{-1/2} = Q \diag(\lambda_1(C^{-1} D_{11}), \dots, \lambda_r(C^{-1} D_{11})) Q^\ast$. Moreover, if $\lambda_k(C^{-1} D_{11}) \ge 1$ for all $1 \le k \le r$, then $D_{-} = D_{11}$. 
\end{corollary}
\begin{lemma}\label{lem:Z}
Let $r \le s$ be positive integers. Given any $(C,D)\in \Herm_r^{++} \times \Herm_s^{++}$, an invertible matrix $Z\in \GL_s(\mathbb{C})$ satisfies conditions 
\begin{itemize}
\item $Z D Z^{\ast} = I_s$,
\item $Z \Omega_{+} (C) Z^\ast = \Omega_+(\Sigma)$ where $\Sigma = \diag( \lambda_{1}(D_{11}^{-1}C), \dots, \lambda_{r}(D_{11}^{-1}C))$,
\end{itemize}
if and only if $Z$ is of the form:
\[
Z = \begin{bmatrix}
P D_{11}^{-1/2} & 0 \\
-Q W D_{12}^\ast D_{11}^{-1} & Q W
\end{bmatrix} = \begin{bmatrix}
P & 0 \\
0 & QW
\end{bmatrix} \begin{bmatrix}
D_{11}^{-1/2} & 0 \\
-D_{12}^\ast D_{11}^{-1} & I_{s-r}
\end{bmatrix}.
\]
Here $Q\in U(r - s), W = (D_{22} - D_{12}^\ast D_{11}^{-1} D_{12})^{-1/2}$, $P \in U(r)$ such that $P^{\ast}  \Sigma P = D_{11}^{-1/2} C D_{11}^{-1/2}$ is an SVD.
\end{lemma}
\begin{proof}
We denote $\lambda_k \coloneqq \lambda_k(D_{11}^{-1} C), 1\le k \le r$. Let $Z$ be an invertible matrix satisfying the two conditions. We partition $Z$ as 
\[
Z = \begin{bmatrix}
Z_{11} & Z_{12} \\
Z_{21} & Z_{22}
\end{bmatrix}.
\] 
We first claim that $Z_{12} = 0$. Indeed, for any $Y\in \Omega_+(C)$, the upper left $r \times r$ submatrix of $ZYZ^\ast$ is 
\begin{equation}\label{lem:Z:Z12}
Z_{11} Y_{11} Z_{11}^\ast + Z_{12} Y_{22} Z_{12}^\ast + Z_{12} Y_{21} Z_{11}^\ast + Z_{11} Y_{12} Z_{12}^\ast \preceq \Sigma.
\end{equation}
We observe that for any $K\in \Herm_r^+$ such that $K \preceq C$, $Y = \begin{bmatrix}
K & 0 \\
0 & c I_{s - r}
\end{bmatrix} \in \Omega_+(C)$ for any $c > 0$. 

Thus \eqref{lem:Z:Z12} becomes 
\[
Z_{11} K Z_{11}^\ast + c Z_{12} Z_{12}^\ast \preceq \Sigma,
\]
which forces $Z_{12}Z_{12}^\ast= 0$ and $Z_{12} = 0$ since $c$ can be arbitrarily large. Therefore, we must have 
\[
Z_{11} C Z_{11}^\ast = \Sigma,\quad Z D Z^\ast = I_s.
\]
We observe that 
\[
\Sigma = Z_{11} C Z_{11}^\ast = Z_{11} D_{11}^{1/2} (D_{11}^{-1/2} C D_{11}^{-1/2})  D_{11}^{1/2}  Z^\ast_{11}= (Z_{11} D_{11}^{1/2} V) \Sigma (Z_{11} D_{11}^{1/2} V)^\ast,
\]
where $D_{11}^{-1/2} C D_{11}^{-1/2} = V \Sigma V^\ast$ is the SVD of $D_{11}^{-1/2} C D_{11}^{-1/2}$. Next we claim that 
\[
\Sigma^{-1/2} Z_{11} D_{11}^{1/2} V \Sigma^{1/2} \in U(r),
\] 
from which we may conclude that 
\[
Z_{11} = P D_{11}^{-1/2},
\]
where $P = \Sigma^{1/2} S \Sigma^{-1/2}   V^\ast $ for some $S\in U(r)$. To prove the claim, we denote $X \coloneqq Z_{11} D_{11}^{1/2} V$. By the relation $\Sigma = X \Sigma X^\ast$ we have 
\[
I_r = \left( \Sigma^{-1/2} X \Sigma^{1/2} \right) \left( \Sigma^{-1/2} X \Sigma^{1/2} \right)^\ast.
\]
This implies that $\Sigma^{-1/2} X \Sigma^{1/2}  \in U(r)$. As a consequence we can write $X =\Sigma^{1/2}   S \Sigma^{-1/2}$ for some $S \in U(r)$. Since $ZDZ^\ast = I_s$, we also have $Z_{11} D_{11} Z_{11}^\ast = I_r$. Thus $P = \Sigma^{1/2} S \Sigma^{-1/2}   V^\ast  \in U(r)$ where $S = \diag(R_{1},\dots, R_{l}) \in U(r)$ and sizes of $R_1,\dots, R_{l}$ are multiplicities of $\lambda_1,\dots, \lambda_r$. Therefore, we may write $P =\diag(R_{1},\dots, R_{l}) V^\ast$.

Next we notice that the lower left $(s - r) \times r$ submatrix of $Z D Z^\ast$ is 
\[
(Z_{21} D_{11} + Z_{22} D_{12}^\ast ) D_{11}^{-1/2} P^\ast = 0.
\]
Therefore $Z_{21} = -Z_{22} D_{12}^\ast D_{11}^{-1}$. Since the lower right $(s - r) \times (s - r)$ submatrix of $Z D Z^\ast$ is 
\[
(Z_{21} D_{12} + Z_{22} D_{22}) Z_{22}^\ast = I_{s - r},
\]
we obtain 
\[
Z_{22}^\ast Z_{22} = (D_{22} - D_{12}^\ast D_{11}^{-1} D_{12} )^{-1}.
\]

Lastly we claim that $Z_{22} =Q W$ where $Q\in U(s-r)$ and $W = (D_{22} - D_{12}^\ast D_{11}^{-1} D_{12} )^{-1/2}$ and this completes the proof. In fact, the claim can be proved by observing 
\[
(Z_{22}W^{-1})^\ast  (Z_{22}W^{-1}) = W^{-1} Z_{22}^\ast Z_{22} W^{-1} = I_{n-m}.
\]
\end{proof}
\begin{proposition}[uniqueness of lift]\label{prop:lift}
Let $r \le s$ be positive integers  and let 
\[
\delta_s: \Herm_s^{++} \times \Herm_s^{++} \to \mathbb{R}
\]
be the divergence of one of the types in Theorem~\ref{thm:divergence}. For any $(C,D)\in \Herm_r^{++} \times \Herm_s^{++}$, there is a unique $C_+\in \Omega_+(C)$ such that 
$Z D Z^\ast = I_s$ and $Z C_+ Z^\ast = \Lambda$ for some matrix $Z$, where
\[
\Lambda = \diag(\min\{1,\lambda_{1}(D_{11}^{-1} C)\}, \dots, \min\{1,\lambda_{m}(D_{11}^{-1} C)\}, 1, \dots, 1).
\]
In particular, we have $\delta_s(C_+, D) = \delta_s(\Omega_+(C),D)$.
\end{proposition}
\begin{proof}
Let $Z$ be an invertible matrix such that $Z D Z^\ast = I_s$ and $Z \Omega_+(C) Z^\ast = \Omega(\Sigma)$ where $\Sigma = \diag(\lambda_1(D_{11}^{-1}C),\dots, \lambda_r(D_{11}^{-1}C))$. By Lemma~\ref{lem:Z}, we conclude that $Z$ is of the form:
\begin{equation}\label{prop:lift:eq1}
Z = \begin{bmatrix}
P D_{11}^{-1/2} & 0 \\
-Q W D_{12}^\ast D_{11}^{-1} & Q W
\end{bmatrix} = \begin{bmatrix}
P & 0 \\
0 & QW
\end{bmatrix} \begin{bmatrix}
D_{11}^{-1/2} & 0 \\
-D_{12}^\ast D_{11}^{-1} & I_{s-r}
\end{bmatrix},
\end{equation}
where $Q\in U(s - r), W = (D_{22} - D_{12}^\ast D_{11}^{-1} D_{12})^{-1/2}$ and $P \in U(r)$ such that 
\[
P^{\ast}  \Sigma P = D_{11}^{-1/2} C D_{11}^{-1/2}
\] 
is an SVD. Therefore, $\delta_s(D,Y) = \delta_s(D, \Omega_+(C))$ if and only if $\delta_s(I_s, Z Y Z^\ast) = \delta_s(I_s, \Omega_+(\Sigma))$. It is sufficient to prove that $\Sigma_+$ is unique for the pair $(I_s,\Sigma)$ and $Z^{-1} \Sigma_+ (Z^{-1})^\ast$ is constant for different choices of $Z$. 
 
According to the definition of $\Sigma_+$, we have 
\[
\lambda_j(\Sigma_+)  =
\begin{cases}
\min\{1,\lambda_{k}(D_{11}^{-1} C)\},~\text{if}~1 \le j \le r, \\
1,~\text{if}~ r + 1 \le j \le s.
\end{cases}
\]
We denote by $\Lambda\in \mathbb{C}^{n\times n}$ the following diagonal matrix:
\[
\diag(\min\{1,\lambda_{1}(D_{11}^{-1} C)\}, \dots, \min\{1,\lambda_{m}(D_{11}^{-1} C)\}, 1, \dots, 1) = \diag(\underbrace{\mu_1,\dots, \mu_1}_{\text{$n_1$ times}},\dots, \underbrace{\mu_{l},\dots, \mu_{l}}_{\text{$n_{l}$ times}}),
\]
where $\mu_1 < \cdots < \mu_{l}$ and $\sum_{j=1}^{l} n_j = s$. Then $\Sigma_+ = R \Lambda R^\ast$ for some $R\in U(s)$. By the same argument as in the proof of Proposition~\ref{prop:projection}, we conclude that $R = \diag(R_1,\dots, R_l)$ where $R_j \in U(n_j), 1 \le j \le l$. Thus $\Sigma_+ = \Lambda$ is unique. 

Moreover, we notice that flexibility of $P$ and $Q$ in \eqref{prop:lift:eq1} is also from multiplicities of $\mu_1,\dots, \mu_l$. 
\begin{equation}\label{prop:lift:eq:C+}
C_+ = Z^{-1} \Sigma_+ (Z^{-1})^\ast
\end{equation}
is unique as well.
\end{proof}
According to Proposition~\ref{prop:projection}, there is a well-defined map 
\[
\pi_{+}: \Herm_r^{++} \times \Herm_s^{++} \to \Herm_s^{++},\quad \pi_{+}(C,D) = C_{+},
\] 
such that the diagram in Figure~\ref{fig:unique lift} commutes. Here $\operatorname{proj}_2$ is the projection map onto the second factor and we denote a divergence on $\Herm_s^+$ and its induced function on $\Herm_r^+ \times \Herm_s^+$
by $\delta_s$ and $\delta_s(\Omega_{+}(\cdot),\cdot)$, respectively. 
\begin{figure}[!hbt]
\centering
  \begin{tikzcd}[row sep=1.5cm, column sep=2cm]
      \Herm_r^{++} \times \Herm_s^{++}  \arrow{d}[left]{ (\pi_{+},\operatorname{proj}_2) } \arrow{r}{ \delta_s(\Omega_{+}(\cdot),\cdot)}  & \mathbb{R}   \\
    \Herm_s^{++}  \times \Herm_s^{++}  \arrow{ru}[below]{\delta_s} &    
    \end{tikzcd}
  \caption{unique projection}
    \label{fig:unique lift}
\end{figure}

\begin{corollary}
Let $r \le s$ be positive integers  and let 
\[
\delta_s: \Herm_s^{++} \times \Herm_s^{++} \to \mathbb{R}
\]
be the divergence of one of the types in Theorem~\ref{thm:divergence}. For any $(C,D)\in \Herm_r^{++} \times \Herm_s^{++}$, we have 
\[
\pi_{+}(C,D) = C_+ = \begin{bmatrix}
D_{11}^{1/2}P^\ast & 0 \\
D_{12}^\ast D_{11}^{-1/2}P^\ast &  (D_{22} - D_{12}^\ast D_{11}^{-1} D_{12})^{1/2}
\end{bmatrix} \Lambda  \begin{bmatrix}
P D_{11}^{1/2} & P D_{11}^{-1/2} D_{12}  \\
0 &  (D_{22} - D_{12}^\ast D_{11}^{-1} D_{12})^{1/2}
\end{bmatrix},
\]
where $\Lambda = \diag(\min\{1,\lambda_{1}(D_{11}^{-1} C)\}, \dots, \min\{1,\lambda_{r}(D_{11}^{-1} C)\}, 1, \dots, 1)$ and $P^\ast \Sigma P = D_{11}^{-1/2} C D_{11}^{-1/2}$ is any SVD of $D_{11}^{-1/2} C D_{11}^{-1/2}$.
\end{corollary}

We conclude this subsection by a remark that in general $\pi_{-}(C,D)$ and $\pi_{+}(C,D)$ depend on both $C$ and $D$, thus $\pi_{-}(C,D) \ne \pi_{-}(C',D), \pi_{+}(C,D) \ne \pi_{+}(C,D')$ if $C \ne C'$ and $D \ne D'$. We prove in Propositions~\ref{prop:projection} and \ref{prop:lift} that $D_{-}\in \Omega_{-}(D)$ and $C_{+} \in \Omega_{+}(C)$ are unique in the sense of their eigenvalues. However, one can find infinitely many elements $X\in \Omega_{-}(D)$ and $Y \in \Omega_{+}(C)$ such that 
\[
\delta_r (C, X) = \delta_r (C, \Omega_{-}(D)),\quad \delta_s (D, Y) = \delta_s (\Omega_{+}(C), D).
\]
For example, we consider 
\[
C = \begin{bmatrix}
1 & 0 \\
0 & 1
\end{bmatrix},\quad D = \begin{bmatrix}
1 & 0 & 0 \\
0 & 1/2 & 0 \\
0 & 0 & 1
\end{bmatrix},\quad \delta(X,Y) = \lVert \log (X^{-1} Y) \rVert.
\]
Then $\delta_2(C, \Omega_{-}(D)) = \delta_3(\Omega_{+}(C),D) = 1$. We let 
\[
X_{\varepsilon} = \begin{bmatrix}
2^{\varepsilon} & 0 \\
0 & 2^{\sqrt{1- \varepsilon^2}}
\end{bmatrix}, \quad \varepsilon \in [0, \sqrt{2}/2].
\]
It is straightforward to verify that $X_{\varepsilon}  \in \Omega_{-}(D)$ and $\delta_2(C, X_{\varepsilon}) = 1$ for each $\varepsilon \in [0, \sqrt{2}/2]$. Similarly, if we let 
\[
Y_{\varepsilon} = \begin{bmatrix}
2^{-\sqrt{-2\varepsilon - \varepsilon^2}} & 0 & 0 \\
0 & 2^{\varepsilon} & 0  \\
0 & 0 & 1
\end{bmatrix},\quad \varepsilon \in [-1,0],
\]
then $Y_{\varepsilon} \in \Omega_+(C)$ and $\delta_3(D, Y_{\varepsilon}) = 1$ for each $\varepsilon \in [-1,0]$
\section{Explicit formulae for \texorpdfstring{$\operatorname{GD}_{d,\delta}$}{}}\label{sec:explicit formulae}
In this section, we consider the calculation of $\operatorname{GD}_{d,\delta}$ with respect to commonly used $d$ and $\delta$. To be more precise, we let $d: \Gr(r,n) \times Gr(s,n)\to \mathbb{R}$ be the function \cite{ye2016schubert} induced by the Asimov distance \cite{DD16}, the Binet-Cauchy distance \cite{WS03}, the chordal distance \cite{JRN96}, the Fubini-Study distance \cite{Fubini1904,Study1905}, the Martin distance \cite{DD16}, the procrustes distance \cite{Yasuko12}, the spectral distance \cite{BN02}, the projection distance \cite{HL08} and the geodesic distance \cite{Wong67}. Moreover, we let $\delta: \Herm^{++}_r \times \Herm^{++}_s \to \mathbb{R}$ be the function induced by one of the divergences in Theorem~\ref{thm:divergence}. 

In Algorithm~\ref{alg:GD}, we present an algorithm to compute $\operatorname{GD}_{d,\delta}$. We remark that Lemma~\ref{lem:stratification} implies that a generic element in $\Herm^+_{n,r} \times \Herm^{+}_{m,s}$ must lie in $K_0$. Thus according to Corollary~\ref{cor:simpify delta_H}, we are able to directly compute $\operatorname{GD}_{d,\delta}(A,B)$ by formulae presented in Tables~\ref{tab:explicit formulae 1}--\ref{tab:explicit formulae 4} for almost any pair $(A,B)\in \operatorname{GD}_{d,\delta}(A,B)$.

\scriptsize{
\begin{table}[htb!]
\centering
\begin{tabular}{|c|c|c|}
\hline
 & $\alpha\beta$ log-det ($\alpha \ne 0, \beta \ne 0, \alpha + \beta \ne 0$) & Stein's loss ($\alpha \ne 0$)   \\
\hline
Asimov    & $\sqrt{\theta_r^2 +  \frac{\left[  \sum\limits_{i=1}^r 
\log \left( \frac{\alpha \lambda^\beta_i + \beta \lambda_i^{-\alpha}}{\alpha + \beta} \right) \right]^2}{\alpha^2 \beta^2 } }$  & $\sqrt{\theta_r^2 +  \frac{\left[ \sum\limits_{i=1}^r 
\left( \lambda_i^{-\alpha} + \log \lambda_i^{\alpha } \right) - r 
 \right]^2}{\alpha^4} }$ \\ \hline

Binet-Cauchy    & $\sqrt{ 1 - \prod\limits_{i=1}^r \cos^2 \theta_i +  \frac{\left[ \sum\limits_{i=1}^r 
\log \left( \frac{\alpha \lambda^\beta_i + \beta \lambda_i^{-\alpha}}{\alpha + \beta} \right) \right]^2}{\alpha^2 \beta^2} }$  & $\sqrt{1 - \prod\limits_{i=1}^r \cos^2 \theta_i +  \frac{\left[ \sum\limits_{i=1}^r 
\left( \lambda_i^{-\alpha} + \log \lambda_i^{\alpha } \right) - r 
 \right]^2}{\alpha^4} }$  \\ \hline

chordal    & $\sqrt{ \sum\limits_{i=1}^r \sin^2 \theta_i +  \frac{\left[ \sum\limits_{i=1}^r 
\log \left( \frac{\alpha \lambda^\beta_i + \beta \lambda_i^{-\alpha}}{\alpha + \beta} \right) \right]^2}{\alpha^2 \beta^2} }$  & $\sqrt{\sum\limits_{i=1}^r \sin^2 \theta_i +  \frac{\left[ \sum\limits_{i=1}^r 
\left( \lambda_i^{-\alpha} + \log \lambda_i^{\alpha } \right) - r 
 \right]^2}{\alpha^4} }$  \\ \hline

Fubini-Study    & $\sqrt{ \arccos^2 \left( \prod\limits_{i=1}^r \cos \theta_i \right)
 +  \frac{\left[ \sum\limits_{i=1}^r 
\log \left( \frac{\alpha \lambda^\beta_i + \beta \lambda_i^{-\alpha}}{\alpha + \beta} \right) \right]^2}{\alpha^2 \beta^2} }$  & $\sqrt{ \arccos^2 \left( \prod\limits_{i=1}^r \cos \theta_i \right) +  \frac{\left[ \sum\limits_{i=1}^r 
\left( \lambda_i^{-\alpha} + \log \lambda_i^{\alpha } \right) - r 
 \right]^2}{\alpha^4} }$  \\ \hline
 
Martin    & $\sqrt{ 2 \sum\limits_{i=1}^r \log \cos \theta_i
 +  \frac{\left[ \sum\limits_{i=1}^r 
\log \left( \frac{\alpha \lambda^\beta_i + \beta \lambda_i^{-\alpha}}{\alpha + \beta} \right) \right]^2}{\alpha^2 \beta^2} }$  & $\sqrt{ 2 \sum\limits_{i=1}^r \log \cos \theta_i +  \frac{\left[ \sum\limits_{i=1}^r 
\left( \lambda_i^{-\alpha} + \log \lambda_i^{\alpha } \right) - r 
 \right]^2}{\alpha^4} }$  \\ \hline
 
 procrustes    & $\sqrt{ 4 \sum\limits_{i=1}^r  \sin^2 \left(\frac{\theta_i}{2}\right)
 +  \frac{\left[ \sum\limits_{i=1}^r 
\log \left( \frac{\alpha \lambda^\beta_i + \beta \lambda_i^{-\alpha}}{\alpha + \beta} \right) \right]^2}{\alpha^2 \beta^2} }$  & $\sqrt{ 4 \sum\limits_{i=1}^r  \sin^2 \left(\frac{\theta_i}{2}\right) +  \frac{\left[ \sum\limits_{i=1}^r 
\left( \lambda_i^{-\alpha} + \log \lambda_i^{\alpha } \right) - r 
 \right]^2}{\alpha^4} }$  \\ \hline

 projection    & $\sqrt{ \sin^2 \theta_r
 +  \frac{\left[ \sum\limits_{i=1}^r 
\log \left( \frac{\alpha \lambda^\beta_i + \beta \lambda_i^{-\alpha}}{\alpha + \beta} \right) \right]^2}{\alpha^2 \beta^2} }$  & $\sqrt{ \sin^2 \theta_r +  \frac{\left[ \sum\limits_{i=1}^r 
\left( \lambda_i^{-\alpha} + \log \lambda_i^{\alpha } \right) - r 
 \right]^2}{\alpha^4} }$  \\ \hline
 
  spectral    & $\sqrt{ 4 \sin^2 \left(\frac{\theta_r}{2}\right)
 +  \frac{\left[ \sum\limits_{i=1}^r 
\log \left( \frac{\alpha \lambda^\beta_i + \beta \lambda_i^{-\alpha}}{\alpha + \beta} \right) \right]^2}{\alpha^2 \beta^2} }$  & $\sqrt{ 4 \sin^2 \left(\frac{\theta_r}{2}\right) +  \frac{\left[ \sum\limits_{i=1}^r 
\left( \lambda_i^{-\alpha} + \log \lambda_i^{\alpha } \right) - r 
 \right]^2}{\alpha^4} }$  \\ \hline
 
  geodesic    & $\sqrt{  \sum\limits_{i=1}^r   \theta^2_i
 +  \frac{\left[ \sum\limits_{i=1}^r 
\log \left( \frac{\alpha \lambda^\beta_i + \beta \lambda_i^{-\alpha}}{\alpha + \beta} \right) \right]^2}{\alpha^2 \beta^2} }$  & $\sqrt{ \sum\limits_{i=1}^r   \theta^2_i +  \frac{\left[ \sum\limits_{i=1}^r 
\left( \lambda_i^{-\alpha} + \log \lambda_i^{\alpha } \right) - r 
 \right]^2}{\alpha^4} }$  \\ \hline
\end{tabular}
\caption{Explicit formulae of $\operatorname{GD}_{d,\delta}$}
\label{tab:explicit formulae 1}
\end{table}} \normalsize

\scriptsize{
\begin{table}[htb!]
\centering
\begin{tabular}{|c|c|c|}
\hline
 & Burg ($\beta \ne 0$)  & Itakura-Saito ($\alpha \ne 0$)  \\
\hline
Asimov   & $\sqrt{\theta_r^2 +\frac{\left[  \sum\limits_{i=1}^r 
\left( \lambda_i^{\beta} -  \log \lambda_i^{\beta} \right) - r 
 \right]^2}{\beta^4} }$  & $\sqrt{\theta_r^2 + \frac{\left[  \sum\limits_{i=1}^r 
\log \left( \frac{\lambda^{-\alpha}_i}{1 +  \log \lambda_i^{-\alpha}}  \right) 
 \right]^2}{\alpha^4} }$  \\ \hline

Binet-Cauchy    & $\sqrt{1 - \prod\limits_{i=1}^r \cos^2 \theta_i + \frac{\left[  \sum\limits_{i=1}^r 
\left( \lambda_i^{\beta} -  \log \lambda_i^{\beta} \right) - r 
 \right]^2}{\beta^4} }$  & $\sqrt{1 - \prod\limits_{i=1}^r \cos^2 \theta_i + \frac{\left[  \sum\limits_{i=1}^r 
\log \left( \frac{\lambda^{-\alpha}_i}{1 +  \log \lambda_i^{-\alpha}}  \right) 
 \right]^2}{\alpha^4} }$  \\ \hline

chordal    & $\sqrt{ \sum\limits_{i=1}^r \sin^2 \theta_i +  \frac{\left[  \sum\limits_{i=1}^r 
\left( \lambda_i^{\beta} -  \log \lambda_i^{\beta} \right) - r 
 \right]^2}{\beta^4} }$  & $\sqrt{\sum\limits_{i=1}^r \sin^2 \theta_i +  \frac{\left[  \sum\limits_{i=1}^r 
\log \left( \frac{\lambda^{-\alpha}_i}{1 +  \log \lambda_i^{-\alpha}}  \right) 
 \right]^2}{\alpha^4} }$  \\ \hline
 
 Fubini-Study    & $\sqrt{ \arccos^2 \left( \prod\limits_{i=1}^r \cos \theta_i \right)
 +  
 \frac{\left[  \sum\limits_{i=1}^r 
\left( \lambda_i^{\beta} -  \log \lambda_i^{\beta} \right) - r 
 \right]^2}{\beta^4}   }$  & $\sqrt{ \arccos^2 \left( \prod\limits_{i=1}^r \cos \theta_i \right) 
 + 
  \frac{\left[  \sum\limits_{i=1}^r 
\log \left( \frac{\lambda^{-\alpha}_i}{1 +  \log \lambda_i^{-\alpha}}  \right) 
 \right]^2}{\alpha^4} }$  \\ \hline
 
 Martin    & $\sqrt{ 2 \sum\limits_{i=1}^r \log \cos \theta_i
 +   
 \frac{\left[  \sum\limits_{i=1}^r 
\left( \lambda_i^{\beta} -  \log \lambda_i^{\beta} \right) - r 
 \right]^2}{\beta^4}  }$  & $\sqrt{ 2 \sum\limits_{i=1}^r \log \cos \theta_i 
 + 
  \frac{\left[  \sum\limits_{i=1}^r 
\log \left( \frac{\lambda^{-\alpha}_i}{1 +  \log \lambda_i^{-\alpha}}  \right) 
 \right]^2}{\alpha^4} 
  }$  \\ \hline

  procrustes    & $\sqrt{ 4 \sum\limits_{i=1}^r  \sin^2 \left(\frac{\theta_i}{2}\right)
 +  
  \frac{\left[  \sum\limits_{i=1}^r 
\left( \lambda_i^{\beta} -  \log \lambda_i^{\beta} \right) - r 
 \right]^2}{\beta^4}  }$  & $\sqrt{ 4 \sum\limits_{i=1}^r  \sin^2 \left(\frac{\theta_i}{2}\right) 
 + 
  \frac{\left[  \sum\limits_{i=1}^r 
\log \left( \frac{\lambda^{-\alpha}_i}{1 +  \log \lambda_i^{-\alpha}}  \right) 
 \right]^2}{\alpha^4}  
   }$  \\ \hline
   
    projection    & $\sqrt{ \sin^2 \theta_r
 +  
   \frac{\left[  \sum\limits_{i=1}^r 
\left( \lambda_i^{\beta} -  \log \lambda_i^{\beta} \right) - r 
 \right]^2}{\beta^4} }$  & $\sqrt{ \sin^2 \theta_r 
 + 
   \frac{\left[  \sum\limits_{i=1}^r 
\log \left( \frac{\lambda^{-\alpha}_i}{1 +  \log \lambda_i^{-\alpha}}  \right) 
 \right]^2}{\alpha^4}  
  }$  \\ \hline

  spectral    & $\sqrt{ 4 \sin^2 \left(\frac{\theta_r}{2}\right)
 +     
 \frac{\left[  \sum\limits_{i=1}^r 
\left( \lambda_i^{\beta} -  \log \lambda_i^{\beta} \right) - r 
 \right]^2}{\beta^4}
}$  & $\sqrt{ 4 \sin^2 \left(\frac{\theta_r}{2}\right) +     \frac{\left[  \sum\limits_{i=1}^r 
\log \left( \frac{\lambda^{-\alpha}_i}{1 +  \log \lambda_i^{-\alpha}}  \right) 
 \right]^2}{\alpha^4}  
 }$  \\ \hline
 
   geodesic    & $\sqrt{  \sum\limits_{i=1}^r   \theta^2_i
 +  
  \frac{\left[  \sum\limits_{i=1}^r 
\left( \lambda_i^{\beta} -  \log \lambda_i^{\beta} \right) - r 
 \right]^2}{\beta^4} 
 }$  & $\sqrt{ \sum\limits_{i=1}^r   \theta^2_i +  \frac{\left[  \sum\limits_{i=1}^r 
\log \left( \frac{\lambda^{-\alpha}_i}{1 +  \log \lambda_i^{-\alpha}}  \right) 
 \right]^2}{\alpha^4}  
 }$  \\ \hline
\end{tabular}
\caption{Explicit formulae of $\operatorname{GD}_{d,\delta}$}
\label{tab:explicit formulae 2}
\end{table}}\normalsize

\scriptsize{
\begin{table}[htb!]
\centering
\begin{tabular}{|c|c|c|}
\hline
  & geodesic  & Kullback-Leibler \\
\hline
Asimov   & $\sqrt{\theta_r^2 +  \sum\limits_{i=1}^r 
\log^2 \lambda_i}$ & $\sqrt{\theta_r^2 +  \left[\frac{ \sum\limits_{i=1}^r  (\lambda_i^{-1} +
\log \lambda_i -1) }{2}\right]^2 }$ \\ \hline

Binet-Cauchy    & $\sqrt{1 - \prod\limits_{i=1}^r \cos^2 \theta_i  +  \sum\limits_{i=1}^r 
\log^2 \lambda_i}$ & $\sqrt{1 - \prod\limits_{i=1}^r \cos^2 \theta_i +  \left[\frac{ \sum\limits_{i=1}^r  (\lambda_i^{-1} +
\log \lambda_i -1)  }{2}\right]^2 }$  \\ \hline

chordal    & $\sqrt{\sum\limits_{i=1}^r \sin^2 \theta_i  +  \sum\limits_{i=1}^r 
\log^2 \lambda_i}$ & $\sqrt{\sum\limits_{i=1}^r \sin^2 \theta_i  +  \left[\frac{ \sum\limits_{i=1}^r  (\lambda_i^{-1} +
\log \lambda_i -1)  }{2}\right]^2 }$ \\ \hline

Fubini-Study    & $\sqrt{ \arccos^2 \left( \prod\limits_{i=1}^r \cos \theta_i \right) +  \sum\limits_{i=1}^r 
\log^2 \lambda_i}$ & $\sqrt{ \arccos^2 \left( \prod\limits_{i=1}^r \cos \theta_i \right) +  \left[\frac{ \sum\limits_{i=1}^r  (\lambda_i^{-1} +
\log \lambda_i -1) }{2}\right]^2 }$  \\ \hline

Martin   & $\sqrt{ 2 \sum\limits_{i=1}^r \log \cos \theta_i +  \sum\limits_{i=1}^r 
\log^2 \lambda_i}$ & $\sqrt{ 2 \sum\limits_{i=1}^r \log \cos \theta_i +  \left[\frac{ \sum\limits_{i=1}^r  (\lambda_i^{-1} +
\log \lambda_i -1) }{2}\right]^2  }$  \\ \hline

procrustes   & $\sqrt{ 4 \sum\limits_{i=1}^r  \sin^2 \left(\frac{\theta_i}{2}\right) +  \sum\limits_{i=1}^r 
\log^2 \lambda_i}$ & $\sqrt{ 4 \sum\limits_{i=1}^r  \sin^2 \left(\frac{\theta_i}{2}\right) +  \left[ \frac{\sum\limits_{i=1}^r  (\lambda_i^{-1} +
\log \lambda_i -1) }{2}\right]^2  }$  \\ \hline

projection   & $\sqrt{ \sin^2 \theta_r+  \sum\limits_{i=1}^r 
\log^2 \lambda_i}$ & $\sqrt{ \sin^2 \theta_r+  \frac{\left[ \sum\limits_{i=1}^r  (\lambda_i^{-1} +
\log \lambda_i -1) \right]^2 }{2} }$  \\ \hline

spectral   & $\sqrt{ 4 \sin^2 \left(\frac{\theta_r}{2}\right) +  \sum\limits_{i=1}^r 
\log^2 \lambda_i}$ & $\sqrt{ 4 \sin^2 \left(\frac{\theta_r}{2}\right) +  \left[\frac{ \sum\limits_{i=1}^r  (\lambda_i^{-1} +
\log \lambda_i -1)  }{2}\right]^2 }$  \\ \hline

geodesic   & $\sqrt{ \sum\limits_{i=1}^r   \theta^2_i +  \sum\limits_{i=1}^r 
\log^2 \lambda_i}$ & $\sqrt{ \sum\limits_{i=1}^r   \theta^2_i +  \left[ \frac{ \sum\limits_{i=1}^r  (\lambda_i^{-1} +
\log \lambda_i -1) }{2} \right]^2  }$ \\ \hline
\end{tabular}
\caption{Explicit formulae of $\operatorname{GD}_{d,\delta}$}
\label{tab:explicit formulae 3}
\end{table}}\normalsize

\scriptsize{
\begin{table}[htb!]
\centering
\begin{tabular}{|c|c|c|}
\hline
  & R\'{e}nyi ($0 < \alpha < 1$)  & Bhattacharyya\\
\hline
Asimov   & $\sqrt{\theta_r^2 + \left[ \frac{\sum\limits\limits_{i=1}^r (\alpha \lambda_i^{\alpha - 1} + (1-\alpha) \lambda_i^\alpha)}{2\alpha (1-\alpha)} \right]^2 }$ & $\sqrt{\theta_r^2 +  
4 \left[ \sum\limits_{i=1}^r \log\left(\frac{\lambda_{i}^{\frac{1}{2}} +\lambda_{i}^{-\frac{1}{2}}}{2} \right) \right]^2 }$ \\ \hline

Binet-Cauchy    & $\sqrt{1 - \prod\limits_{i=1}^r \cos^2 \theta_i  +  \left[ \frac{\sum\limits_{i=1}^r (\alpha \lambda_i^{\alpha - 1} + (1-\alpha) \lambda_i^\alpha)}{2\alpha (1-\alpha)} \right]^2 }$ & $\sqrt{1 - \prod\limits_{i=1}^r \cos^2 \theta_i +  
4 \left[ \sum\limits_{i=1}^r \log\left(\frac{\lambda_{i}^{\frac{1}{2}} +\lambda_{i}^{-\frac{1}{2}}}{2} \right) \right]^2 }$  \\ \hline

chordal    & $\sqrt{\sum\limits_{i=1}^r \sin^2 \theta_i  +  \left[ \frac{\sum\limits_{i=1}^r (\alpha \lambda_i^{\alpha - 1} + (1-\alpha) \lambda_i^\alpha)}{2\alpha (1-\alpha)} \right]^2 }$ & $\sqrt{\sum\limits_{i=1}^r \sin^2 \theta_i  +  
4 \left[ \sum\limits_{i=1}^r \log\left(\frac{\lambda_{i}^{\frac{1}{2}} +\lambda_{i}^{-\frac{1}{2}}}{2} \right) \right]^2 }$ \\ \hline

Fubini-Study    & $\sqrt{ \arccos^2 \left( \prod\limits_{i=1}^r \cos \theta_i \right) +  \left[ \frac{\sum\limits_{i=1}^r (\alpha \lambda_i^{\alpha - 1} + (1-\alpha) \lambda_i^\alpha)}{2\alpha (1-\alpha)} \right]^2 }$ & $\sqrt{ \arccos^2 \left( \prod\limits_{i=1}^r \cos \theta_i \right) +  
4 \left[ \sum\limits_{i=1}^r \log\left(\frac{\lambda_{i}^{\frac{1}{2}} +\lambda_{i}^{-\frac{1}{2}}}{2} \right) \right]^2 }$  \\ \hline

Martin   & $\sqrt{ 2 \sum\limits_{i=1}^r \log \cos \theta_i +  \left[ \frac{\sum\limits_{i=1}^r (\alpha \lambda_i^{\alpha - 1} + (1-\alpha) \lambda_i^\alpha)}{2\alpha (1-\alpha)} \right]^2 }$ & $\sqrt{ 2 \sum\limits_{i=1}^r \log \cos \theta_i +  
4 \left[ \sum\limits_{i=1}^r \log\left(\frac{\lambda_{i}^{\frac{1}{2}} +\lambda_{i}^{-\frac{1}{2}}}{2} \right) \right]^2  }$  \\ \hline

procrustes   & $\sqrt{ 4 \sum\limits_{i=1}^r  \sin^2 \left(\frac{\theta_i}{2}\right) +  \left[ \frac{\sum\limits_{i=1}^r (\alpha \lambda_i^{\alpha - 1} + (1-\alpha) \lambda_i^\alpha)}{2\alpha (1-\alpha)} \right]^2 }$ & $\sqrt{ 4 \sum\limits_{i=1}^r  \sin^2 \left(\frac{\theta_i}{2}\right) +  
4 \left[ \sum\limits_{i=1}^r \log\left(\frac{\lambda_{i}^{\frac{1}{2}} +\lambda_{i}^{-\frac{1}{2}}}{2} \right) \right]^2  }$  \\ \hline

projection   & $\sqrt{ \sin^2 \theta_r +  \left[ \frac{\sum\limits_{i=1}^r (\alpha \lambda_i^{\alpha - 1} + (1-\alpha) \lambda_i^\alpha)}{2\alpha (1-\alpha)} \right]^2 }$ & $\sqrt{ \sin^2 \theta_r +  4 \left[ \sum\limits_{i=1}^r \log\left(\frac{\lambda_{i}^{\frac{1}{2}} +\lambda_{i}^{-\frac{1}{2}}}{2} \right) \right]^2 }$  \\ \hline

spectral   & $\sqrt{ 4 \sin^2 \left(\frac{\theta_r}{2}\right) +  \left[ \frac{\sum\limits_{i=1}^r (\alpha \lambda_i^{\alpha - 1} + (1-\alpha) \lambda_i^\alpha)}{2\alpha (1-\alpha)} \right]^2 }$ & $\sqrt{ 4 \sin^2 \left(\frac{\theta_r}{2}\right) +  
4 \left[ \sum\limits_{i=1}^r \log\left(\frac{\lambda_{i}^{\frac{1}{2}} +\lambda_{i}^{-\frac{1}{2}}}{2} \right) \right]^2 }$  \\ \hline

geodesic   & $\sqrt{ \sum\limits_{i=1}^r   \theta^2_i +  \left[ \frac{\sum\limits_{i=1}^r (\alpha \lambda_i^{\alpha - 1} + (1-\alpha) \lambda_i^\alpha)}{2\alpha (1-\alpha)} \right]^2 }$ & $\sqrt{ \sum\limits_{i=1}^r   \theta^2_i +  
4 \left[ \sum\limits_{i=1}^r \log\left(\frac{\lambda_{i}^{\frac{1}{2}} +\lambda_{i}^{-\frac{1}{2}}}{2} \right) \right]^2  }$ \\ \hline
\end{tabular}
\caption{Explicit formulae of $\operatorname{GD}_{d,\delta}$}
\label{tab:explicit formulae 4}
\end{table}}\normalsize

\begin{algorithm}[!htbp]
\caption{Algorithm for $\operatorname{GD}_{d,\delta}$}
\label{alg:GD}
\begin{algorithmic}[1]
\renewcommand{\algorithmicrequire}{\textbf{Input}}
\Require
$A\in \Herm^+_{n,r}, B\in \Herm^+_{m,s}$
\renewcommand{\algorithmicensure}{\textbf{Output}}
\Ensure
$\operatorname{GD}_{d,\delta}(A,B)$
\If {$n \ge m$}
\State set $A_1 = A$, $B_1 = \begin{bmatrix}
B & 0 \\
0 & 0
\end{bmatrix}$;
\Else
\State set $A_1 = \begin{bmatrix}
A & 0 \\
0 & 0
\end{bmatrix}$, $B_1 = B$;
\EndIf
\State compute compact SVD: $A_1 = U_A \Sigma_A U^\ast_A$;
\State compute compact SVD: $B_1 = U_B \Sigma_B U^\ast_B$;
\State compute SVD: $U_A^\ast U_B = P \Sigma Q^\ast$; \Comment{diagonal elements of $\Sigma$: $\sigma_1 \ge \cdots \ge \sigma_{\min\{r,s\}}$}
\State compute $l = \# \{i: \sigma_i = 0\}$;
\If {$l = 0$}
\State compute $C = P^\ast \Sigma_A P, D=Q^\ast \Sigma_B Q$; \Comment{Corollary~\ref{cor:simpify delta_H}}
\If {$r \ge s$}
\State compute $\theta_i = \arccos \sigma_i, 1\le i \le s$;
\State compute eigenvalues of $D^{-1}C_{11}$: $\mu_1 \ge \cdots \mu _s$;  \Comment{$C_{11}$: upper-left $s \times s$ submatrix of $C$}
\State set $\lambda_i = \max\{1,\mu_i\}, 1 \le i \le s$;
\State compute $\operatorname{GD}_{d,\delta}(A,B)$ by formula in Tables~\ref{tab:explicit formulae 1}--\ref{tab:explicit formulae 3};
\Else
\State compute $\theta_i = \arccos \sigma_i, 1\le i \le r$;
\State compute eigenvalues of $C^{-1}D_{11}$: $\mu_1 \ge \cdots \mu _r$;  \Comment{$D_{11}$: upper-left $s \times s$ submatrix of $D$}
\State set $\lambda_i = \max\{1,\mu_i\}, 1 \le i \le r$;
\State compute $\operatorname{GD}_{d,\delta}(A,B)$ by formula in Tables~\ref{tab:explicit formulae 1}--\ref{tab:explicit formulae 3};
\EndIf
\Else
\If {$s \ge r$}
\State compute 
$\delta = \max_{T\in U(s - r +l)} \delta \left( 
 C, 
\begin{bmatrix}
I_{r-l} & 0 \\
0 & T
\end{bmatrix} D \begin{bmatrix}
I_{r-l} & 0 \\
0 & T^\ast
\end{bmatrix}
\right)$.
\Else
\State compute 
$\delta = \max_{T\in U(r - s +l)} \delta \left( 
\begin{bmatrix}
I_{r-l} & 0 \\
0 & T
\end{bmatrix} C \begin{bmatrix}
I_{r-l} & 0 \\
0 & T^\ast
\end{bmatrix}, 
 D 
\right)$.
\EndIf
\State compute $\operatorname{GD}_{d,\delta}(A,B) = \sqrt{d^2 (\operatorname{span}(U_A), \operatorname{span}(U_B) ) + \delta^2 }$; \Comment{Proposition~\ref{prop:simpify delta_H}}
\EndIf
\State \Return $\operatorname{GD}_{d,\delta}(A,B)$. \end{algorithmic}
\end{algorithm}

\section{\texorpdfstring{$\operatorname{GD}_{d,\delta}$}{} as a distance function}\label{sec:distance}
We denote by $\mathcal{k}^{\infty \times \infty}$ the space of infinite matrices over $\mathcal{k}$ with finitely many nonzero elements and we define 
\[
\Herm_{\infty,r}^+ \coloneqq \left\lbrace
A \in \mathcal{k}^{\infty \times \infty}: \rank (A) = r, A^\ast = A, A \succeq 0
\right\rbrace.
\]
In other words, $\Herm_{\infty,r}^+$ consists of infinite Hermitian positive semidefinite matrices with finitely many nonzero elements. We recall from \eqref{eq:iota star} that for any $r \le m \le n$ there is an inclusion $\iota_{m,n}^{\ast}: \Herm_{m,r}^+ \hookrightarrow \Herm_{n,r}^+$ defined by $\iota_{m,n}^\ast (A) = \begin{bmatrix}
A & 0 \\
0 & 0
\end{bmatrix}$. Thus $\Herm_{\infty,r}^+ $ is simply the direct limit of the direct system $\{ (\{\Herm_{r,m}^{+}\}_{m \ge r}, \{\iota^\ast_{m,n}\}_{n \ge m \ge r})\}$. We further define 
\[
\Herm_{\infty,\infty}^+ \coloneqq \bigsqcup_{r = 0}^\infty \Herm_{\infty,r}^+ .
\]
By definition, we clearly have $\Herm_{\infty,\infty}^+ = \bigcup_{n=1}^{\infty} \bigsqcup_{r=0}^n \Herm_{n,r}^{+}$.

For simplicity, we let $d,\delta$ be as in Section~\ref{sec:explicit formulae}. In this section, we briefly discuss to what extent $\operatorname{GD}_{d,\delta}:\Herm_{\infty,\infty}^+ \times \Herm_{\infty,\infty}^+ \to \mathbb{R}$ is a distance function. 

For $A\in \Herm_{\infty,r}^+$ and $B \in \Herm_{\infty,s}^+$, without loss of generality, we may assume that $r \le s$, $A\in \Herm_{n,r}^+$ and $B \in \Herm_{n,s}^+$. The combination of \eqref{eq:GD} and Proposition~\ref{prop:simpify delta_H} leads to
\begin{equation}\label{eq:GDdistance}
\operatorname{GD}_{d,\delta}(A,B)  = \left( d^2(\mathbb{A},\mathbb{B}) + 
\max_{Y\in \mathcal{B}}  \delta^2(M_{A|\mathbb{A}}(\mathbf{u},\mathbf{v}),Y)
 \right)^{\frac{1}{2}}.
\end{equation}
Here $\mathbb{A} = \pi_{n,r} (A) \in \Gr(r,n), \mathbb{B} = \pi_{n,s} (B) \in \Gr(s,n)$ and 
\begin{equation*}
\mathcal{B} \coloneqq \left\lbrace
\begin{bmatrix}
I_{r - l} & 0 \\
0 & T
\end{bmatrix} M_{B|_{\mathbb{B}}}(\mathbf{u},\mathbf{v}) \begin{bmatrix}
I_{r - l} & 0 \\
0 & T^\ast
\end{bmatrix}, T\in U(s - r + l)
\right\rbrace,
\end{equation*}
where $l = \dim (\mathbb{A} \cap \mathbb{B}^\perp)$, $(\mathbf{u},\mathbf{v})$ is any fixed set of principal vectors between $\mathbb{A}$ and $\mathbb{B}$, $M_{A|_{\mathbb{A}}}(\mathbf{u},\mathbf{v})$ (resp. $M_{B|_{\mathbb{B}}}(\mathbf{u},\mathbf{v})$) is the matrix representation of the operator $A|_{\mathbb{A}}$ (resp. $B|_{\mathbb{B}}$) with respect to $(\mathbf{u},\mathbf{v})$. From \eqref{eq:GDdistance}, we have $\operatorname{GD}_{d,\delta}(A,B) = 0$ if and only if $\mathbb{A} \subseteq \mathbb{B}$ and $M_{A|_{\mathbb{A}}}(\mathbf{u},\mathbf{v})$ is the upper left submatrix of all elements in $\mathcal{B}$. We also notice that if $\mathbb{A} \subseteq \mathbb{B}$, then $l = 0$ and we may require that 
\[
(\mathbf{u}, \mathbf{v}) = (\{u_1,\dots, u_r\}, \{u_1,\dots, u_r, v_{r+1},\dots, v_s \}).
\]
Thus we obtain the proposition that follows. 
\begin{proposition}\label{prop:GDmetric}
Let $d,\delta$ be as in Section~\ref{sec:explicit formulae} and let $A,B,\mathbb{A},\mathbb{B}$ and $l$ be as above. Then we have  
\begin{enumerate}[(i)]
\item\label{prop:GDmetric:item0} If $r \ne s$ then 
\[
\operatorname{GD}_{d,\delta}(A,B) = \operatorname{GD}_{d,\delta}(B,A).
\]
\item\label{prop:GDmetric:item1} If $r = s$ and $l = 0$ then 
\[\operatorname{GD}^2_{d,\delta}(A,B) + \operatorname{GD}^2_{d,\delta}(B,A) = 2d^2(\mathbb{A},\mathbb{B}) + \delta_r^2(M_{A|_{\mathbb{A}}}(\mathbf{u},\mathbf{v}),M_{B|_{\mathbb{B}}}(\mathbf{u},\mathbf{v})),\]
where $\delta_r$ is the divergence in Theorem~\ref{thm:divergence} inducing $\delta$.
\item\label{prop:GDmetric:item2}$\operatorname{GD}_{d,\delta}(A,B) = 0$ if and only if $\mathbb{A} \subseteq \mathbb{B}$ and $A =\rho \circ B|_{\mathbb{A}}$, where $\rho:\mathbb{B} \to \mathbb{A}$ is the orthogonal projection from $\mathbb{B}$ onto $\mathbb{A}$.
\end{enumerate}
\end{proposition}
\begin{proof}
Both \eqref{prop:GDmetric:item0} and \eqref{prop:GDmetric:item1} easily follow from the definition. To prove \eqref{prop:GDmetric:item2}, we may assume that $r \le s$ so the discussion above implies that $\operatorname{GD}_{d,\delta}(A,B) = 0$ if and only if $\mathbb{A} \subseteq \mathbb{B}$ and $M_{A|\mathbb{A}}(\mathbf{u},\mathbf{v})$ is the upper left $r\times r$ submatrix of all elements in $\mathcal{B}$, where $\mathbf{u} = \{u_1,\dots, u_r\}$ is an orthonormal basis of $\mathbb{A}$ and $\mathbf{v} = \{u_1,\dots, u_r, v_{r+1},\dots, v_s \}$ is an orthonormal basis of $\mathbb{B}$. But $l = 0$ indicates that the upper left $r\times r$ submatrix of any element in $\mathcal{B}$ is the same as that of $M_{B|\mathbb{B}}(\mathbf{u},\mathbf{v})$, which is the matrix representation of $\rho \circ B|_{\mathbb{A}}$ with respect to the basis $\mathbf{u}$.
\end{proof}
\begin{remark}
The triangle inequality does not always hold for $\operatorname{GD}_{d,\delta}$. For example, we let both $d$ and $\delta$ be induced by geodesic distances. Let $r,s,t$ be positive integers.
\begin{itemize}
\item If $r = s = t$, then clearly we have 
\[
\operatorname{GD}_{d,\delta}(A,B) + \operatorname{GD}_{d,\delta}(B,C) \ge \operatorname{GD}_{d,\delta}(A, B)
\]
for any $A\in \Herm_{\infty,r}^{+}, B\in \Herm_{\infty,s}^{+}$ and $C\in \Herm_{\infty,t}^{+}$.
\item If $s < r < t$. Then for any $X\in \Herm_{t-s}^{++}$, we have 
\[
\operatorname{GD}_{d,\delta}(I_r, I_s) + \operatorname{GD}_{d,\delta} \left( I_s, \begin{bmatrix}
I_s & 0 \\
0 & X
\end{bmatrix} \right) = 0, 
\]
while for $X = \operatorname{diag}(x_1,\dots, x_{t-s})\in \Herm_{t-s}^{++}$ where $x_1 \ge \cdots \ge x_{t-s} > 0$
\[
\operatorname{GD}_{d,\delta} \left( I_r, \begin{bmatrix}
I_s & 0 \\
0 & X
\end{bmatrix} \right) =\sqrt{\sum_{j=1}^{r-s} \max\{0,\log x_j\}^2}.
\]
Hence the triangle inequality fails if $x_1 > 1$.
\end{itemize}
\end{remark}

\bibliographystyle{abbrv}
\bibliography{mybib1}
\end{document}